\documentclass[10pt,draft]{article}
\usepackage{amsmath,amsthm,amssymb,amsfonts}
\usepackage{enumerate}
\usepackage{mathrsfs}
\usepackage[cmtip,all]{xy}

\numberwithin{equation}{section}
\numberwithin{figure}{section}

\normalsize

\newtheoremstyle{theoremstyle}
  {10pt}      
  {5pt}       
  {\itshape}  
  {}          
  {\bfseries} 
  {:}         
  {.5em}      
  {}          

\newtheoremstyle{examplestyle}
  {10pt}      
  {5pt}       
  {}          
  {}          
  {\bfseries} 
  {:}         
  {.5em}      
  {}          

\theoremstyle{theoremstyle}
\newtheorem{theorem}{Theorem}[section]
\newtheorem*{theorem*}{Theorem}
\newtheorem{lemma}[theorem]{Lemma}
\newtheorem{proposition}[theorem]{Proposition}
\newtheorem*{proposition*}{Proposition}
\newtheorem{corollary}[theorem]{Corollary}
\newtheorem*{corollary*}{Corollary}

\newtheorem{definition}[theorem]{Definition}
\newtheorem{definition*}{Definition}
\newtheorem{remark}[theorem]{Remark}
\newtheorem{remark*}{Remark}

\newcommand{\sfA}{\mathsf{A}}
\newcommand{\sfB}{\mathsf{B}}

\newcommand{\E}{\mathsf{E}}
\newcommand{\M}{\mathsf{M}}
\newcommand{\sfN}{\mathsf{N}}

\newcommand{\D}{\mathscr{D}}

\newcommand{\caC}{{\mathcal C}}
\newcommand{\caD}{{\mathcal D}}
\newcommand{\caE}{{\mathcal E}}

\newcommand{\caN}{{\mathcal N}}

\newcommand{\caQ}{{\mathcal Q}}
\newcommand{\caT}{{\mathcal T}}
\newcommand{\caH}{{\mathcal H}}
\newcommand{\caM}{{\mathcal M}}
\newcommand{\MM}{{\mathscr M}}

\newcommand{\Char}{\mathsf{char}}

\newcommand{\KGL}{\mathsf{KGL}}
\newcommand{\KQ}{\mathsf{KQ}}
\newcommand{\MGL}{\mathsf{MGL}}
\newcommand{\MU}{\mathsf{MU}}
\newcommand{\kgl}{\mathsf{kgl}}
\newcommand{\kq}{\mathsf{kq}}

\newcommand{\MZ}{\mathsf{M}\mathbb{Z}}

\newcommand{\unit}{\mathbf{1}}
\newcommand{\Z}{\mathbb{Z}}
\newcommand{\N}{\mathbb{N}}

\newcommand{\Sm}{\mathbf{Sm}}
\newcommand{\SH}{\mathbf{SH}}
\newcommand{\Ho}{\mathbf{Ho}}

\newcommand{\Mod}{\mathrm{Mod}}

\newcommand{\Hom}{\mathrm{Hom}}
\newcommand{\map}{\mathrm{map}}

\newcommand{\eff}{\mathbf{eff}}
\newcommand{\Veff}{\mathbf{Veff}}

\newcommand{\hocolim}{\mathrm{hocolim}}

\newcommand{\sSet}{\mathsf{sSet}}
\newcommand{\Spt}{\mathsf{Spt}}
\newcommand{\scK}{\mathscr{K}}
\newcommand{\scS}{\mathscr{S}}

\newcommand{\Coll}{\mathsf{Coll}}
\newcommand{\Oper}{\mathsf{Oper}}

\newcommand{\Alg}{\mathsf{Alg}}
\newcommand{\Ass}{\mathit{\mathcal{A}ss}}
\newcommand{\Com}{\mathit{\mathcal{C}om}}
\newcommand{\modu}{\mathit{\mathcal{M}od}}
\newcommand{\LMod}{\mathit{\mathcal{L}Mod}}

\newcommand{\Mor}{\mathit{\mathcal{M}or}}
\newcommand{\Map}{\mathrm{Map}}
\newcommand{\End}{\mathrm{End}}

\title{{\bf Motivic slices and colored operads}}
\author{Javier J. Guti\'errez, Oliver R\"ondigs, Markus Spitzweck, Paul Arne {\O}stv{\ae}r}
\date{December, 2010}
\begin{document}
\maketitle
\begin{abstract}
Colored operads were introduced in the 1970's for the purpose of studying homotopy invariant algebraic structures on topological spaces.
In this paper we introduce colored operads in motivic stable homotopy theory.
Our main motivation is to uncover hitherto unknown highly structured properties of the slice filtration.
The latter decomposes every motivic spectrum into its slices, 
which are motives, 
and one may ask to what extend the slice filtration preserves highly structured objects such as algebras and modules.
We use colored operads to give a precise solution to this problem.
Our approach makes use of axiomatic setups which specialize to classical and motivic stable homotopy theory.
Accessible $t$-structures are central to the development of the general theory.
Concise introductions to colored operads and Bousfield (co)localizations are given in separate appendices.
\end{abstract}

\newpage
{\small{\tableofcontents}}
\newpage

\section{Introduction}
\label{section:introduction}

Our main goal in this paper is to prove preservation of algebra and module structures in the slice filtration 
\begin{equation*}
\cdots
\subset
\Sigma^{q+1}_{T}\SH(S)^{\eff}
\subset
\Sigma^{q}_{T}\SH(S)^{\eff}
\subset
\Sigma^{q-1}_{T}\SH(S)^{\eff}
\subset
\cdots
\end{equation*}
introduced by Voevodsky \cite{voevodsky.open}. 
Here $\SH(S)^{\eff}$ denotes the effective motivic stable homotopy category of a separable noetherian base scheme $S$.
The basic idea behind the slice filtration is to decompose motivic spectra into computable pieces called the slices.
One may therefore view the slice filtration as an analogue of the Postnikov tower in stable homotopy theory. 
In order to encode and control the behavior of highly structured algebras and modules in the slice filtration, 
we employ techniques afforded by colored operads. 
Although the motivic setup is our main motivation,  
we develop an axiomatic setup which allows for a wider use of our main results.
This was prompted by the current interest in slice techniques, 
as exemplified by the solution of the Kervaire invariant one problem \cite{HHR}.

Knowledge of the slices of a motivic spectrum $\E$ provides valuable input in Voevodsky's slice spectral sequence abutting to the 
motivic homotopy groups of $\E$. 
This fact has motivated considerable work on the slices of algebraic $K$-theory $\KGL$, 
algebraic cobordism $\MGL$, the sphere spectrum $\unit$ and hermitian $K$-theory $\KQ$.
We shall use these examples to illustrate some applications of our main results.
Motivic twisted $K$-theory introduced in \cite{SOmotivivtwisted} provides another major motivation for the writing of this paper.
Indeed, 
the constructions of the Atiyah-Hirzebruch type of spectral sequences for motivic twisted $K$-theory are based on results for motivic 
$E_{\infty}$-algebras which are shown here in greater generality.

By \cite{pelaez}, \cite{spitzweck.slices-landweber}, 
the slice filtration can be viewed as a lax symmetric monoidal functor for the motivic stable homotopy category
\begin{equation}
\label{equation:totalslice}
s_*
\colon
\SH(S) 
\longrightarrow 
\SH(S)^\Z.
\end{equation}
Up to homotopy, 
this shows that $s_*$ preserves algebras, commutative algebras, and modules over such algebras.
Suppose $\E$ is underlying cofibrant and the unit $\unit\to\E$ is a map between effective motivic symmetric ring spectra which 
induces an isomorphism on the $0$-slices.
In \cite{pelaez} it is shown that the slices of every motivic spectrum are highly structured $\E$-modules.
When the base scheme is a field of characteristic zero, 
it follows that the slices are motives, 
i.e.~strict modules over the motivic Eilenberg-Mac\,Lane spectrum $\MZ$ \cite{Rondigs-Ostvar1}, \cite{Rondigs-Ostvar2}.
However, 
in general, 
the techniques in \cite{pelaez} appear to be impractical to answer questions concerning the preservation of algebra 
and module structures.
In this paper we shall develop a theory which enables such a treatment.
Colored operads is now a well-developed branch of abstract homotopy theory with many applications. 
One of the main innovations of this paper is to establish a link between colored operads and motivic homotopy theory.
This allows to discuss preservation of algebras and modules in the slice filtration.
The setup in which we shall prove our main results is comprised of the homotopy category of a stable combinatorial simplicial symmetric 
monoidal model category relative to either a pair of full subcategories or an integer indexed family of full subcategories,
which are subject to a few natural intertwining axioms.

Next we describe in more details the sections of this paper and the manner in which its aims are achieved.
To begin with,
Section \ref{section:axiomatics} introduces the axiomatic setups alluded to earlier in this introduction.
Inspired by the slice filtration in motivic homotopy theory we contemplate a well-behaved homotopy category $\caC$ relative to either a 
pair of full subcategories or an integer indexed family of full subcategories.
Every accessible $t$-structure on $\caC$ gives rise to colocalization and localization functors corresponding to the full subcategories.
These functors, denoted by $c_i$ for colocalization and $l_i$ for localization, 
give rise to the slice functors $s_i$ and the slice towers, 
which for any object $X$ of $\caC$ and $i\in\Z$ comprise exact triangles 
\begin{equation*}
c_{i+1}X
\rightarrow
c_{i}X
\rightarrow
s_{i}X
\rightarrow
c_{i+1}X[1]
\end{equation*}
in the homotopy category.
On the level of model categories, 
the localization functors are modelled by fibrant replacements and the colocalization functors by cofibrant replacements.  
For a pair $\caD_{1}\subset\caD_{0}$ of full subcategories we introduce its core,
which can be identified with the localization of $\caD_{0}$ along $l_{1}$-equivalences.
Here $l_{1}$ denotes the localization functor corresponding to $\caD_{1}$.
It turns out that many properties of highly structured objects are related to properties of cores.
For example, 
the core associated to the module category of an $E_{\infty}$-algebra is a symmetric monoidal category.
In Section \ref{jkmbdsjwgvkj} we give examples of the axiomatic setup arising in classical and motivic stable homotopy theory.

We defer the precise formulations of our main results to Section \ref{hfgdfg}. 
What follows is a sample of applications for motivic spectra.
Over a base field of characteristic zero, 
the mod-$p$ motivic Eilenberg-Mac\,Lane spectrum $\mathsf{M}\mathbb{Z}/p$ is an $E_{\infty}$-algebra whose underlying object is cellular, 
and there is an isomorphism 
$$
s_*(\mathsf{M}\mathbb{Z}/p \wedge \mathsf{M}\mathbb{Z}/p) 
\cong 
\mathsf{M}\mathbb{Z}/p \wedge_{s_* \unit} \mathsf{M}\mathbb{Z}/p.
$$
Thus the dual mod-$p$ rigid Steenrod algebra is isomorphic to a derived tensor product over the total slice $s_*\unit$ of the sphere spectrum.
We show that $s_*\unit$ is an $E_{\infty}$-algebra in $\SH(S)^\Z$ and (\ref{equation:totalslice}) factors through its derived category $\D(s_*\unit)$.
If $\E$ is a motivic $E_\infty$-ring spectrum, 
then its ``connective cover'' $f_0 \E$ and $0$-slice $s_0\E$ are also motivic $E_\infty$-ring spectra.
Moreover, 
for any $\E$-module $\MM$, 
its $(q-1)$-connective cover $f_q \MM$ is an $f_0\E$-module and its $q$th slice $s_q \MM$ is an $s_0\E$-module.
Over a perfect field, 
$s_0\unit=\MZ$ by work of Levine \cite{Levine-slices} and Voevodsky \cite{Voevodsky-zeroslices}.
The $E_\infty$-structure furnished on $s_0\unit$ by our methods coincides with the usual one on $\MZ$.
In fact the $E_\infty$-structure on $\MZ$ is unique over perfect fields.
For $\kgl=f_{0}\KGL$, 
connective algebraic $K$-theory, 
there is a naturally induced map between $E_\infty$-ring spectra 
$$
\kgl 
\to 
\MZ,
$$
which is reminiscent of the Chern class for complex vector bundles.
The analogue for $\kq=f_{0}\KQ$, 
connective hermitian $K$-theory,
is a naturally induced map between $E_\infty$-ring spectra 
$$
\kq 
\to 
\MZ \vee \bigvee_{i<0} \Sigma^{2i,0}\MZ/2,
$$
which is reminiscent of the total Stiefel-Whitney class for real vector bundles.
In this example we assume the base scheme is a perfect field $k$ of characteristic $\Char(k)\neq 2$ and $\sqrt{-1}\in k$.

Detailed proofs of our main results are written in Section \ref{hfgdfg}.
Preparatory material on algebras is given in Section \ref{lkmmhgh}.
Additional examples of interest in motivic homotopy theory are treated in Section \ref{asgfcdj}.

The language of colored operads is admittedly not well-known to practitioners of motivic homotopy theory.
For the convenience of the reader we have included Appendix \ref{jnhgd} on colored operads and algebras, 
and Appendix \ref{ytrgvv} on Bousfield localizations.
In some of the proofs we find it convenient to employ the machinery of $\infty$-categories \cite{Lurie}.

\section{Axiomatic settings}
\label{section:axiomatics}

In this section we describe axiomatic settings in which we shall prove statements about preservation of highly structured algebras and modules.
For precursors on $t$-structures we refer the reader to \cite{BBD}.
We begin by recalling some model categorical definitions.

A model category $\caM$ is called \emph{stable} if it is pointed and the suspension operator is invertible in the homotopy category $\Ho(\caM)$. 
If $\caM$ is stable, 
then $\Ho(\caM)$ is triangulated and acquires products and coproducts over arbitrary index sets derived from those of $\caM$. 
A \emph{combinatorial} model category is a locally presentable and cofibrantly generated model category.
A model category is called \emph{proper} if weak equivalences are preserved by pullback along fibrations and by pushout along cofibrations.

\subsection{A graded family of subcategories}
\label{hygrf}

Let $\caM$ be a stable combinatorial simplicial symmetric monoidal model category. 
Assume that $\caM$ is proper and its unit for the monoidal structure is cofibrant. 
Denote by $\caC=\Ho(\caM)$ the homotopy category of $\caM$.
Let $\{\caC_i\mid i\in\Z\}$ be a family of full subcategories of $\caC$ satisfying the following axioms 
(which are easily checked in practice):
\begin{enumerate}
\item[(A1)] The subcategory $\caC_i$ is closed under isomorphisms for every $i\in\Z$.
\item[(A2)] $\caC_{i+1} \subset \caC_i$ for every $i \in \Z$.
\item[(A3)] Every $\caC_i$ is generated by homotopy colimits and extensions by a set of objects $\mathscr{K}_i$, 
i.e., $\caC_i$ is the smallest full subcategory of $\caC$ that contains $\mathscr{K}_i$ and is closed under homotopy colimits and extensions.
\item[(A4)] The tensor unit is in $\caC_0$.
\item[(A5)] If $X \in \caC_i$ and $Y \in \caC_j$, then $X \otimes Y \in \caC_{i+j}$.
\end{enumerate}

An object $X$ in $\caC$ is called \emph{effective} if it lies in $\caC_0$.
The subcategories $\caC_i$ are not necessarily closed under desuspensions.
This makes closure under extensions an interesting condition.

\begin{proposition} \label{htgrfd}
Suppose $\caN$ is a stable combinatorial model category and $\caH \subset \Ho(\caN)$ a full subcategory which is generated by homotopy colimits 
and extensions by a set of objects. 
Then $\caH$ is the homologically non-negative part of a $t$-structure on $\Ho(\caN)$. 
Moreover, there exists a set of objects that generates $\caH$ by homotopy colimits.
\end{proposition}
\begin{proof}
Since $\caN$ is a stable combinatorial model category its associated $\infty$-category is stable and presentable; 
for the latter, see \cite[Remark 5.5.1.5, Proposition A.3.7.6]{Lurie}.
Recall that every presentable $\infty$-category is generated by homotopy colimits by a set of objects.
Thus our claim follows from \cite[Proposition 16.1]{lurie.DAGI} which asserts the existence of such a $t$-structure.
\end{proof}

The $t$-structures of $\Ho(\caN)$ arising in this way are called \emph{accessible $t$-structures} in \cite{lurie.DAGI}.
It follows from Proposition \ref{htgrfd} that every $\caC_i$ is the homologically non-negative part of some $t$-structure on $\caC$. 
Denote the corresponding colocalization and localization functors by $c_i$ and $l_i$, respectively. 
Thus, 
for any $X$ in $\caC$ and every $i\in\Z$, 
there is an exact triangle
\begin{equation*}
c_i X \longrightarrow X \longrightarrow l_i X \longrightarrow c_i X[1].
\end{equation*}
The $t$-structures obtained in this way for the slice filtration on the motivic stable homotopy category are degenerate. 
On the other hand, 
every accessible $t$-structure $\caC_{\ge 0}$ on $\caC$ (recall that $t$-structures are determined by their homologically non-negative parts)
gives rise to subcategories $\caC_i \subset \caC$ satisfying (A1)--(A3) by setting $\caC_i = \caC_{\ge 0}[i]$.

Suppose that $\caN$ is a stable combinatorial proper model category and $\caH \subset \Ho(\caN)$ is the homologically non-negative part of an 
accessible $t$-structure on $\Ho(\caN)$.
Denote by $c$ and $l$ the corresponding colocalization and localization functors.
Fix a set of objects $\mathscr{K}$ of $\caH$ that generates $\caH$ by homotopy colimits. 
In this situation we may define two new model structures on $\caN$, 
namely the colocalization $C^{\mathscr{K}} \caN$ with respect to $\mathscr{K}$ and the localization $L_{\mathscr{S}}\caN$ with respect to 
$\mathscr{S}=\{0\to K \mid K\in\mathscr{K}\}$ (cf.~Appendix~\ref{ytrgvv}).
The colocalized and localized model structures exist by~\cite{Barwick}, \cite{Hirschhorn}.

\begin{lemma}
On the level of model categories, 
$c$ is modeled by cofibrant replacement in $C^{\mathscr{K}}\caN$ and $l$ is modeled by fibrant replacement in $L_{\mathscr{S}}\caN$.
\end{lemma}
\begin{proof}
The first part follows because $c$ is right adjoint to the given inclusion $\caH \subset \Ho(\caN)$. 
The second part follows from the next lemma.
\end{proof}

\begin{lemma} \label{jyhtgd}
Let $c'$ be the colocalization functor associated with a set of objects $\mathscr{K}'$ of $\Ho(\caN)$.
For $X$ in $\caN$, 
complete the colocalization map $c'X \to X$ to an exact triangle
\begin{equation*}
c'X \longrightarrow X  \longrightarrow C  \longrightarrow c'X[1].
\end{equation*}
Then $X \to C$ is an $\mathscr{S}'$-local equivalence where $\mathscr{S}'=\{0 \to K' \mid K' \in \mathscr{K}'\}$. 
Moreover, 
if the subcategory of $\mathscr{K}'$-colocal objects of $\Ho(\caN)$ is closed under extensions, 
then $C$ is $\mathscr{S}'$-local and hence $X \to C$ is an $\mathscr{S}'$-localization.
\end{lemma}
\begin{proof}
First we show that $X \to C$ is an $\mathscr{S}'$-local equivalence. 
Let $Y$ be an $\mathscr{S}'$-local object and consider the fiber sequence of mapping spaces
$$\map(C,Y) \longrightarrow \map(X,Y) \longrightarrow \map(c'X,Y).$$
The simplicial set $\map(c'X,Y)$ is contractible since $Y$ is $\mathscr{S}'$-local and $c'X$ is a homotopy colimit of objects of $\mathscr{K}'$. 
This proves the first part of the lemma.

To prove the second part, 
complete the composite map $f \colon c'C \to C \to c'X[1]$ to an exact triangle
$$
c'X \longrightarrow X' \longrightarrow c'C \longrightarrow c'X[1].
$$
Since $\Ho(\caN)$ is a triangulated category there exists a map $h \colon X' \to X$ such that
$$
\xymatrix{c'X \ar[r] \ar[d] & X' \ar[r] \ar@{.>}[d]^h & c'C \ar[r]^f \ar[d] & c'X[1] \ar[d] \\
c'X \ar[r] & X \ar[r] & C \ar[r] & c'X[1]}
$$
is a map of exact triangles. 
Now both $c'X$ and $c'C$ are $\mathscr{K}'$-colocal.
Thus if the $\mathscr{K}'$-colocal objects are closed under extensions, 
it follows that $X'$ is also $\mathscr{K}'$-colocal.

Next we show $h$ is a $\mathscr{K}'$-colocal equivalence.
Let $K$ be a non-negative suspension of an object of $\mathscr{K}'$.
We need to show that $\Hom(K,h)$ is an isomorphism. 
Consider the diagram of abelian groups:
$$
\xymatrixcolsep{1pc}
\xymatrix{\Hom(K[1],c'C) \ar[r] \ar[d] & \Hom(K,c'X) \ar[r] \ar[d] & \Hom(K,X') \ar[r] \ar[d] &
\Hom(K,c'C) \ar[r] \ar[d] & \Hom(K,c'X[1]) \ar[d] \\
\Hom(K[1],C) \ar[r] & \Hom(K,c'X) \ar[r] & \Hom(K,X) \ar[r] &
\Hom(K,C) \ar[r] & \Hom(K,c'X[1])}
$$
The second and fifth vertical maps are identity maps, 
while the first and fourth maps are isomorphisms because $K$ is $\mathscr{K}'$-colocal. 
By the five lemma, 
$h$ is a $\mathscr{K}'$-colocal equivalence.
Hence $h$ is a $\mathscr{K}'$-colocalization; 
that is, 
$c'X \to X'$ is an isomorphism and thus $c'C=0$. 
It follows that $\map(K,C)=\map(K,c'C)$ is contractible, and $C$ is therefore $\mathscr{S}'$-local.
\end{proof}

Thus our functors $c_i$ and $l_i$ can be modeled on the model category level by cofibrant and fibrant replacement functors.
For $l_i$ we can refine this by localizing at the set of maps $\mathscr{S}=\{0 \to K\mid K \in \mathscr{K}\}$, 
where $\mathscr{K}$ generates the homologically non-negative part of a $t$-structure by homotopy colimits and extensions.
We note the following useful result.

\begin{lemma} \label{jyhdht}
Suppose $\caH \subset \Ho(\caN)$ is a full subcategory generated by homotopy colimits and extensions by a set $\widetilde{\mathscr{K}}$. 
Then the localization functor $l$ of the corresponding $t$\nobreakdash-structure is the $\widetilde{\mathscr{S}}$-localization functor for the set
$\widetilde{\mathscr{S}}=\{0 \to K \mid K \in \widetilde{\mathscr{K}}\}$.
\end{lemma}
\begin{proof}
Let $\mathscr{K}$ be a set that generates $\caH$ by homotopy colimits. 
By Lemma~\ref{jyhtgd},
$l$ is the $\mathscr{S}$\nobreakdash-localization functor for $\mathscr{S}=\{0 \to K \mid K \in \mathscr{K}\}$.
It suffices to show that the $\widetilde{\mathscr{S}}$\nobreakdash-local and the $\mathscr{S}$-local objects coincide. 
Of course any $\mathscr{S}$-local object is $\widetilde{\mathscr{S}}$-local.
The converse is proved by observing that if $\map(K_1,X)$ and $\map(K_2,X)$ are contractible, 
and $K$ is an extension of $K_1$ and $K_2$, 
then $\map(K,X)$ is contractible.
This follows by arguing with long exact sequences for the given extension.
\end{proof}

\begin{remark}
Lemmas \ref{jyhtgd} and \ref{jyhdht} hold true without the properness assumption on~$\caN$. 
The colocalization and localization functors exist because they exist on the level of $\infty$-categories, 
cf.~the proof of Proposition \ref{htgrfd} and \cite[Proposition 16.1]{lurie.DAGI}.
\end{remark}

\subsection{The core of a pair of subcategories}
\label{jytgef} 

Let $\caN$ be a stable combinatorial model category with homotopy category $\caD=\Ho(\caN)$.
Fix full subcategories $\caD_1\subset\caD_0$ of $\caD$ and assume that $\caD_i$ is the homologically non-negative part of an accessible $t$-structure 
on $\caD$. 
By a {\em pair} we mean $(\caD_0,\caD_1)$ with these properties.
For $\caD_i$ we denote the corresponding localization functor by $l_i$ and the colocalization functor by $c_i$ (for $i=0,1$). 
Hence for any object $X$ in $\caD$ there are exact triangles
$$
c_i(X) \longrightarrow X \longrightarrow l_i(X) \longrightarrow c_i(X)[1].
$$
Let $\caE_i$ be the essential image of $l_i$.
With these definitions the factorization $\caD \to \caE_i$ of $l_i$ is left adjoint to the inclusion $\caE_i \subset \caD$.
Define \emph{the core of the pair} $(\caD_0,\caD_1)$ as the intersection $\caQ=\caE_1\cap\caD_0$.
Informally we may think of the core as the quotient of $\caD_0$ by $\caD_1$. 
\begin{lemma} \label{jyhgd}
The restriction of $l_1$ to ${\caD_0}$ factors through $\caD_0$. 
Dually, 
the restriction of $c_0$ to ${\caE_1}$ factors through $\caE_1$.
\end{lemma}

\begin{proof}
For $X$ in $\caD_0$ the localization $l_1(X)$ is the cofiber of $c_1(X)\to X$. 
Since $\caD_1 \subset \caD_0$ it follows that $c_1(X)\in\caD_0$. 
This allows us to conclude since $\caD_0$ is closed under homotopy colimits.
Alternatively, note that $\Hom(l_1(X),Y)=0$ for every $Y\in\caE_0$ since $\caE_0\subset\caE_1$ implies $\Hom(l_1X,Y)\to\Hom(X,Y)=0$ is an isomorphism. 
The second statement follows by a dual argument.
\end{proof}

Lemma \ref{jyhgd} shows there exists an induced functor $L \colon \caD_0 \to \caQ$ which is left adjoint to the inclusion $\caQ \subset \caD_0$.
\begin{corollary}
The functor $L$ identifies $\caQ$ with the localization of $\caD_0$ with respect to $l_1$-local equivalences.
\end{corollary}

\begin{lemma}
If a set $\mathscr{K}$ generates $\caD_1$ by homotopy colimits and extensions, 
then $\caQ$ is the localization of $\caD_0$ with respect to the set of maps $\mathscr{S}=\{0\to K\mid K\in\mathscr{K}\}$.
\end{lemma}
\begin{proof}
Lemma \ref{jyhtgd} shows the $\mathscr{S}$-local objects are exactly the $l_1$-local objects. 
But the full subcategory of $\caD_0$ consisting of $\mathscr{S}$-local objects is exactly the localization of $\caD_0$ along $\mathscr{S}$. 
(Note that $\caD_0$ has a left proper combinatorial model since it is the homotopy category of a presentable $\infty$-category,
see \cite[Proposition 16.1]{lurie.DAGI} and \cite[Proposition A.3.7.6]{Lurie}.)
\end{proof}

We define the \emph{slice functor} of $\caD$ relative to the pair $(\caD_0, \caD_1)$ by setting 
$$
s(X)=l_1(c_0(X)).
$$ 
Thus, 
for every $X$ in $\caD$, 
there is an exact triangle
\begin{equation}
\label{equation:slice}
c_1(X) \longrightarrow c_0(X) \longrightarrow s(X) \longrightarrow c_1(X)[1].
\end{equation}
The exact triangle (\ref{equation:slice}) determines the slice $s(X)$ up to unique isomorphism:
Uniqueness follows because $\caD_1$ is closed under suspension and the Hom group $\Hom_{\caD}(\caD_1,s(X))$ is trivial;
the vanishing is a consequence of the universal property of the $l_1$-localization functor.
\begin{lemma}
For every object $X$ of $\caD$, 
$s(X)\in\caQ$ and $c_0(l_1(X))\in\caQ$.
\end{lemma}
\begin{proof}
This follows from Lemma~\ref{jyhgd}.
\end{proof}

\begin{lemma} \label{hbgfs}
The functors $s$ and $c_0 \circ l_1$ are canonically isomorphic.
\end{lemma}
\begin{proof} 
The proof expands on techniques originating in \cite{BBD}, \cite[Proposition 6.10]{lurie.DAGI}.
We begin by contemplating the diagram:
\begin{equation}
\label{equation:BBDluriediagram}
\xymatrix{
c_0(X) \ar[r] \ar[d] & X \ar[r] & l_1(X)  \\
l_1(c_0(X)) \ar@{.>}[rr]^{\varphi_X} & & c_0(l_1(X)) \ar[u]  }
\end{equation}
By reference to Lemma \ref{jyhgd}, 
combined with the universal properties of $c_0$ and $l_1$,
there exists a unique dotted arrow $\varphi_X$ rendering (\ref{equation:BBDluriediagram}) commutative.
Clearly, 
for varying $X$, 
the $\varphi_X$'s define a natural transformation.
In the rest of the proof we argue that $\varphi_X$ is an isomorphism. 
It suffices to show there is a naturally induced isomorphism
$$
\Hom(Y,l_1(c_0(X))) 
\longrightarrow 
\Hom(Y,c_0(l_1(X))) 
\overset{\cong}{\longrightarrow} 
\Hom(Y,l_1(X))
$$
for every object $Y$ in $\caD_0$. 
In effect, 
consider the induced map of long exact sequences:
$$
\xymatrix{\Hom(Y,c_1(c_0(X))) \ar[r] \ar[d] & \Hom(Y,c_1(X)) \ar[d] \\
\Hom(Y,c_0(X)) \ar[r] \ar[d] & \Hom(Y,X) \ar[d] \\
\Hom(Y,l_1(c_0(X))) \ar[r] \ar[d] & \Hom(Y,l_1(X)) \ar[d] \\
\Hom(Y,c_1(c_0(X))[1]) \ar[r] \ar[d] & \Hom (Y,c_1(X)[1]) \ar[d] \\
\Hom (Y,c_0(X)[1]) \ar[r] & \Hom(Y,X[1]) }
$$
Since $c_1(c_0(X))\cong c_1(X)$ the first and fourth horizontal maps are isomorphisms. 
Since $Y\in\caD_0$ the second horizontal map is an isomorphism and the fifth horizontal map is injective. 
Thus applying the five lemma finishes the proof.
\end{proof}
\begin{remark}
As a motivation for Lemma \ref{hbgfs} the referee pointed out the analogy with chain complexes equipped with the upper and lower truncation functors.
\end{remark}

In what follow we also assume that $\caN$ is a symmetric monoidal model category with a cofibrant unit, 
and the following compatibility axioms relating the derived tensor product on $\caD$ with the full subcategories $\caD_0$ and $\caD_1$:
\begin{enumerate}
\item[(B1)] $\caD_0$ contains the unit and is closed under the tensor product.
\item[(B2)] $\caD_1$ is a tensor ideal in $\caD_0$, i.e., $X\otimes Y\in D_1$ if $X\in \caD_0$ and $Y\in \caD_1$.
\end{enumerate}

\begin{lemma} \label{hgfss}
Let $\caT$ be a triangulated category with homologically non-negative part $\caT_{\ge 0}$ and corresponding localization functor $l$ 
projecting $\caT$ onto $\caT_{\le -1}$.
Suppose
$$
X \longrightarrow Y \longrightarrow Z \longrightarrow X[1]
$$
is an exact triangle and $X\in \caT_{\ge 0}$. 
Then $Y \to Z$ is an $l$-local equivalence.
\end{lemma}
\begin{proof}
Taking Homs out of the exact triangle into an $l$-local object verifies the claim.
\end{proof}

\begin{lemma} \label{bfgdhs}
Suppose $X\to Y$ is an $l_1$-local equivalence in $\caD$ and $Z$ is an object of $\caD_0$.
Then $X\otimes Z\to Y\otimes Z$ is an $l_1$-local equivalence.
\end{lemma}
\begin{proof}
It suffices to prove the lemma for $X\to Y$ the $l_1$-localization map of $X$. 
There is an exact triangle
$$
c_1(X) 
\longrightarrow 
X 
\longrightarrow 
Y 
\longrightarrow 
c_1(X)[1],
$$
where $c_1(X)\in\caD_1$.
Tensoring with $Z$ yields another exact triangle
$$
c_1(X)\otimes Z 
\longrightarrow 
X\otimes Z 
\longrightarrow 
Y\otimes Z 
\longrightarrow 
c_1(X)[1]\otimes Z.
$$
Since $\caD_1\subset\caD_0$ is a tensor ideal, 
$c_1(X) \otimes Z$ is in $\caD_1$.
Thus, 
by Lemma~\ref{hgfss}, 
$X\otimes Z\to Y\otimes Z$ is an $l_1$-local equivalence.
\end{proof}

\begin{corollary} \label{kjuyse}
There is an induced symmetric monoidal structure on $\caQ$ such that the localization functor $\caD_0 \to \caQ$ is symmetric monoidal.
\end{corollary}
\begin{proof}
This follows directly from Lemma \ref{bfgdhs}.
\end{proof}

Later in the paper we shall refine the slice functor $s$ by defining a lax symmetric monoidal functor $\caD\to\Ho(\Mod(l_1(\unit)))$.
This will be obtained by first taking the functor $\caD\rightarrow\caD_0$ and second applying Proposition \ref{dddghhh}.
In order for this to succeed we need to ensure that $l_1(\unit)$ is a ring object in an appropriate sense,
cf.~Proposition \ref{jyssrf}.

\subsection{The graded slice functor}
\label{graded_slice}
In the following we use the setup of Section~\ref{hygrf}:
$\caC$ is the homotopy category of a stable combinatorial simplicial symmetric monoidal proper model category $\caM$, 
and $\{\caC_i\mid i\in\Z\}$ is a collection of full subcategories satisfying the axioms (A1)--(A5). 
Denote by $l_i$ and $c_i$ the localization and colocalization functors associated to $\caC_i$.

Lemma \ref{hbgfs} shows there is a canonical isomorphisms of functors 
\begin{equation}
\label{equation:slicefunctor}
l_{i+1}\circ c_i 
\cong 
c_i\circ l_{i+1}.
\end{equation}
for every $i\in\Z$. 
The \emph{$i$th slice functor $s_i$} refers to either of the isomorphic functors in (\ref{equation:slicefunctor}).
For every $X$ in $\caC$ there is an exact triangle in $\caC$
$$
c_{i+1} X \longrightarrow c_i X \longrightarrow s_i X \longrightarrow c_{i+1} X[1].
$$

We refer to the functor $s_* \colon \caC \to \caC^\Z$ defined by $s_*(X)(k)=s_k(X)$ as the \emph{full slice functor}. 
In the remaining of this section we show that $s_*$ is lax symmetric monoidal when $\caC^\Z$ is equipped with the tensor structure furnished 
by Day's convolution product~\cite{Day}; 
this is,
for $X^{\bullet}$, $Y^{\bullet}$ in $\caC^{\Z}$ and $k\in\Z$,  
$$
(X^{\bullet}\otimes Y^{\bullet})(k)=\coprod_{i+j=k}X(i)\otimes Y(j).
$$
Let $\caQ_i$ denote the core of the pair $(\caC_i,\caC_{i+1})$. 
The following lemma implies that the pairing $\otimes\colon\caC_i\times\caC_j\to\caC_{i+j}$ descends to a pairing 
$\otimes\colon\caQ_i\times\caQ_j\to\caQ_{i+j}$ since $\caQ_i$ is the localization of $\caC_i$ with respect to the $l_{i+1}$-equivalences.
\begin{lemma} \label{ngfsd}
Suppose $X \to Y$ is an $l_{i+1}$-equivalence in $\caC$ and $Z$ is an object of $\caC_j$.
Then $X \otimes Z \to Y \otimes Z$ is an $l_{i+j+1}$-equivalence.
\end{lemma}
\begin{proof}
The argument is analogous to the proof of Lemma \ref{bfgdhs}.
\end{proof}
We deduce that $\prod_{i \in \Z} \caQ_i$ acquires a tensor product turning $\prod_{i\in\Z}\caC_i\to\prod_{i\in\Z}\caQ_i$ into a 
symmetric monoidal functor. 
The latter has a right lax symmetric monoidal adjoint $\prod_{i\in\Z}\caQ_i\to\prod_{i\in\Z}\caC_i$. 
Moreover, 
the symmetric monoidal functor $\prod_{i\in\Z}\caC_i\to\caC^\Z$ has a right lax symmetric monoidal adjoint $\caC^\Z\to\prod_{i\in\Z}\caC_i$.
The full slice functor is the composite
\begin{equation}
\label{equation:kjhfjhg}
\caC 
\longrightarrow 
\caC^\Z 
\longrightarrow 
\prod_{i \in \Z} \caC_i 
\longrightarrow 
\prod_{i \in \Z} \caQ_i 
\longrightarrow 
\prod_{i \in \Z} \caC_i
\longrightarrow 
\caC^\Z.
\end{equation}
The diagonal embedding $\caC\to\caC^\Z$ is lax symmetric monoidal because it is right adjoint to the symmetric monoidal sum functor 
$\caC^\Z\to\caC$.
As noted above, 
the other functors in (\ref{equation:kjhfjhg}) are either symmetric monoidal or lax symmetric monoidal.
Thus $s_*$ is lax symmetric monoidal.

We wish to advocate a simplification of this argument by merging an integer indexed family of subcategories into a pair of subcategories.
As we shall demonstrate, 
this point of view is bootstrapped for proving localization and colocalization results in the graded setup.

Next we let $\caN=\caM^{\Z}$ so that $\caD=\Ho(\caN)=\caC^{\Z}$. 
Then $\caN$ is a symmetric monoidal model category with the monoidal product given by Day's convolution product. 
We define the pair of subcategories $(\caD_0,\caD_1)$ of $\caD$ by
$$
\caD_0
=
\prod_{i \in \Z} \caC_i \subset \caD 
\quad\mbox{and}\quad 
\caD_1
=
\prod_{i \in \Z} \caC_{i+1} \subset \caD.
$$
In the products we consider the subcategories $\caC_i$ of $\caD_0$ and $\caC_{i+1}$ of $\caD_1$ in degree $i\in\Z$.
With these definitions, 
$\caD_1 \subset \caD_0$ and $\caD_i$ is the homologically non-negative part of an accessible $t$-structure on $\caD$. 
Moreover, 
the subcategories $\caD_0$ and $\caD_1$ satisfy the axioms (B1)--(B2).

The core $\caQ$ of this pair, 
as described in Section~\ref{jytgef}, 
is $\prod_{i\in\Z}\caQ_i$. 
From Corollary~\ref{kjuyse} we may infer that $\caQ$ has a symmetric monoidal structure and the localization functor $\caD_0\to\caQ$ is 
symmetric monoidal. 
The full slice functor $s_*$ is given by composing the diagonal embedding $\caC\to\caC^\Z=\caD$ with the slice functor associated to the pair 
$(\caD_0,\caD_1)$, 
as defined in Section~\ref{jytgef}.

\section{Examples}
\label{jkmbdsjwgvkj}
Next we discuss some motivational examples of the axiomatic setup in Section \ref{section:axiomatics}.

\subsection{Motivic slice filtration} \label{jjhhhh}
First we recover the motivic slice filtration \cite{Voevodsky-slice}. 
Let $S$ be a base scheme and $\caM$ the category $\Spt_T^\Sigma(S)$ of motivic symmetric spectra with the stable model structure from 
\cite[\S 4]{jardine.symmetric}.
This is a combinatorial symmetric monoidal proper model category. 
Denote the motivic stable homotopy category $\Ho(\Spt_T^{\Sigma}(S))$ of $S$ by $\caC=\SH(S)$.
The set of objects 
$$
\{\Sigma^\infty_T X_+ \mid X \in \Sm/S\}
$$
generates a full localizing subcategory $\SH(S)^\eff$ of $\SH(S)$, 
where $\Sm/S$ denotes the category of smooth schemes of finite type over $S$ and $T$ is the Tate object.
We set $\caC_i=\Sigma_T^i \SH(S)^\eff$ for every integer $i\in\Z$.

Define $\scK$ as the set of isomorphism classes of the objects $\Sigma^\infty_T X_+[n]$ of $\SH(S)$, 
$X\in \Sm/S$ and $n\in\Z$, 
and set $\scK(q)=\Sigma^{0,q} \scK$ (the $(0,q)$th motivic suspension) for $q\in\Z$. 
In model categorical terms,  
the cofibrant replacement functor in $C^{\scK(q)}\Spt_T^\Sigma(S)$ models the colocalization functor $c_q$ associated to $\caC_q$.
We shall follow the notation in \cite{Voevodsky-slice} where the colocalization functor is denoted by $f_q$.
Define $\scS$ as the set of isomorphism classes of maps $0\to K$ where $K\in\scK$, 
and set $\scS(q)=\Sigma^{0,q} \scS$. 
The fibrant replacement functor in $L_{\scS(q)} \Spt_T^\Sigma(S)$ models the cofiber $l_q$ of the natural transformation $c_q\to\mathrm{Id}$. 

The slice functor $s_q$ is obtained by first applying a colocalization with respect to $\scK(q)$, 
viewing the resulting object as an object in the original category, 
and then localizing this object with respect to $\scS(q+1)$.
By applying our general slice filtration machinery to this example, 
we recover the motivic slice filtration introduced by Voevodsky in \cite{voevodsky.open}.

This setup generalizes easily to categories of modules over effective $E_\infty$-ring spectra, 
for example the motivic Eilenberg-Mac\,Lane spectrum over a perfect field.

\subsection{Motivic $S^1$-filtration}
This is analogous to the previous example.
Denote by $\caC=\Ho(\Spt_{S^1}^{\Sigma}(S))=\SH_{S^1}(S)$ the homotopy category of motivic $S^1$-spectra for the simplicial circle.
A filtration is obtained by setting $\caC_i=\SH_{S^1}(S)$ for $i\le 0$ and $\caC_i=\Sigma_T^i \SH_{S^1}(S)$ for $i>0$.
This can naturally be viewed as an $\N$-indexed family.
By work of Levine \cite{levineS1slices}, 
the associated motivic $S^1$-slices, 
except for the $0$th one, 
admit a further filtration whose layers are in a natural way the Eilenberg-Mac\, Lane spectra associated to a homotopy invariant 
complex of Nisnevich sheaves with transfers.

\subsection{The very effective motivic slice filtration}
Let $\SH(S)^\Veff$ be the full subcategory of $\SH(S)$ generated by homotopy colimits and extensions by the set of objects 
$\{\Sigma^\infty X_+\mid X\in\Sm/S\}$. 
This category was introduced in \cite{SOmotivivtwisted} for the purpose of showing strong convergence of spectral sequences abutting to motivic 
twisted $K$-groups.
The very effective filtration of $\SH(S)$ is defined by setting $\caC_i = \Sigma_T^i \SH(S)^\Veff$. 
If the base scheme is a field which admits a complex embedding, 
for example a number field,
then taking complex points maps $\caC_i$ to its $2i$th topological counterpart $\caC_{2i}^{\textrm{top}}$ comprising $2i$th Postnikov stages.
In this sense the very effective motivic slice filtration is more closely related to Postnikov towers in topology than the motivic slice filtration.

\subsection{Monoidal $t$-structures}

Suppose $\caM$ and $\caC$ are as in Section~\ref{hygrf} and $\caC_{\ge 0}$ is an accessible $t$-structure on $\caC$. 
We set $\caC_i = \caC_{\ge 0} [i]$ for $i\in \Z$. 
If $\caC_{\ge 0}$ is a symmetric monoidal subcategory then the family $\{\caC_i\mid i\in\Z\}$ satisfies the axioms (A1)--(A5).

In particular, 
this applies to the homotopy $t$-structure on $\SH(S)$ generated by the set $\{\Sigma^{p,p}\Sigma^\infty X_+\mid p\in\Z\mbox{ and }X\in\Sm/S\}$.
It is compatible with the smash product.

The usual $t$-structure on the topological stable homotopy category $\SH$ satisfies these properties. 
In this case the slices are Eilenberg-Mac\,Lane spectra on the stable homotopy groups of the spectrum. 
The full slice is the graded stable GEM of a given spectrum.

Letting $\caM$ be $S^{1}$-spectra of simplicial abelian groups we recover the example when $\caC$ is the derived category of abelian groups.

The subcategory $\SH(S)^\Veff\subset\SH(S)$ is another example of an accessible $t$-structure which is compatible with the smash product.

\section{Preservation of algebras}
\label{lkmmhgh}
In this section we transfer $\mathcal{O}$-algebra structures, 
where $\mathcal{O}$ is a cofibrant $C$-colored operad, 
along augmented and coaugmented functors on simplicial monoidal model categories. 
The main results are Theorems \ref{aug_alg} and \ref{coaug_alg}. 
Specializations to localization and colocalization functors and to ring and module structures play an important role in the proofs of the results in 
Section~\ref{hfgdfg}. 
Precursors on colored operads and their algebras are collected in Appendix~\ref{jnhgd}.

We consider $C$-colored operads in the category of simplicial sets acting on a simplicial monoidal model category $\mathcal{M}$. 
Recall that a monoidal model category $\mathcal{M}$ is called \emph{simplicial} if it is a simplicial model category and the simplicial action 
$\otimes$ commutes with the monoidal product $\wedge$ in $\mathcal{M}$, 
i.e., 
there are natural coherent isomorphisms
$$
A\otimes (X\wedge Y)\cong (A\otimes X)\wedge Y
$$
for every simplicial set $A$ and $X$, $Y\in\mathcal{M}$.

It is worthwhile to emphasize the distinction between the monoidal model category of simplicial sets, 
in which our colored operads take values, and the monoidal model category $\mathcal{M}$ on which they act. 
If $\mathcal{O}$ is a $C$\nobreakdash-colored operad in the category of simplicial sets and $\mathcal{M}$ is a simplicial monoidal model category, 
then an $\mathcal{O}$\nobreakdash-algebra ${\bf X}=(X(c))_{c\in C}$ is an object of $\mathcal{M}^C$ equipped with a map of $C$\nobreakdash-colored 
operads $\mathcal{O}\to {\rm End}({\bf X})$ in simplicial sets,
where ${\rm End}({\bf X})$ is defined as
$$
{\rm End}({\bf X})(c_1,\ldots, c_n;c)=\Map(X(c_1)\wedge\cdots\wedge X(c_n), X(c)).
$$
Here $\Map(-,-)$ denotes the simplicial enrichment of~$\mathcal{M}$.
We note that this is consistent with the definition of $\mathcal{O}$-algebras given in Appendix~\ref{jnhgd}, 
since there is a bijection of sets:
$$
\xymatrix{
\sSet(\mathcal{O}(c_1,\ldots, c_n;c), \Map(X(c_1)\wedge\cdots\wedge X(c_n), X(c))) \ar[d] \\
\mathcal{M}(\mathcal{O}(c_1,\ldots, c_n; c)\otimes (X(c_1)\wedge\cdots\wedge X(c_n)), X(c))  }
$$

Two $\mathcal{O}$-algebra structures $\gamma,\gamma'\colon\mathcal{O}\to\End(\mathbf{X})$ on $\mathbf{X}$ in $\mathcal{M}^C$ 
\emph{coincide up to homotopy} if  $\gamma$ and $\gamma'$ are equal in the homotopy category of $C$-colored operads in simplicial sets.

Let $(Q_1, \varepsilon_1),\ldots, (Q_n,\varepsilon_n)$ be augmented functors on $\mathcal{M}$. 
Let $J_1,\ldots, J_n \subseteq C$ be such that $J_i\cap J_k=\emptyset$ if $i\ne k$.  
The \emph{extension of $Q_1,\ldots, Q_n$ over $\mathcal{M}^C$ relative to $J_1,\ldots, J_n$} is the augmented functor $(Q,\varepsilon)$ given by 
$Q\mathbf{X}=(Q_cX(c))_{c\in C}$, 
where $Q_c=Q_i$ if $c\in J_i$ and $Q_c={\rm Id}$ otherwise. 
The augmentation $(\varepsilon_{\mathbf{X}})_c$ is defined as $(\varepsilon_i)_{X(c)}$ if $c\in J_i$ and the identity map if $c\not\in J_i$.
Similarly, 
if $(R_1,\eta_1),\ldots, (R_n,\eta_n)$ are coaugmented functors on $\mathcal{M}$, 
then the \emph{extension of $R_1,\ldots, R_n$ over $\mathcal{M}^C$ relative to $J_1,\ldots, J_n$} is the coaugmented functor $(R,\eta)$ given by 
$R\mathbf{X}=(R_cX(c))_{c\in C}$, 
where $R_c=R_i$ if $c\in J_i$ and $R_c={\rm Id}$ otherwise. 
The coaugmentation $(\eta_{\mathbf{X}})_c$ is defined as $(\eta_i)_{X(c)}$ if $c\in J_i$ and the identity map if $c\not\in J_i$.

\begin{theorem}
\label{aug_alg}
Let $(Q_1,\varepsilon_1),\ldots, (Q_n, \varepsilon_n)$ be augmented functors on a simplicial monoidal model category $\mathcal{M}$. 
Let $\mathcal{O}$ be a cofibrant $C$-colored operad in simplicial sets and consider the extension $(Q,\varepsilon)$ of $Q_1,\ldots,Q_n$ over 
$\mathcal{M}^C$ relative to $J_1,\ldots, J_n\subseteq C$. 
Let $\mathbf{X}$ be an $\mathcal{O}$-algebra and suppose $\varepsilon_{\mathbf{X}}$ induces a trivial fibration of $C$-colored collections
$$
\End_{\mathcal{O}}(Q\mathbf{X})\longrightarrow \Hom_{\mathcal{O}}(Q\mathbf{X},\mathbf{X}).
$$
(For precise definitions see Subsections \ref{endp} and \ref{homp}.)
Then $Q\mathbf{X}$ admits a (homotopy unique) $\mathcal{O}$\nobreakdash-algebra structure such that $\varepsilon_{\mathbf{X}}$ is a map between 
$\mathcal{O}$-algebras.
\end{theorem}
\begin{proof}
Consider the pullback diagram in the category of $C$-colored collections of simplicial sets:
$$
\xymatrix{
\End_{\mathcal{O}}(\varepsilon_{\mathbf{X}})\ar@{.>}[r]^{\rho} \ar@{.>}[d]_{\tau} & \End_{\mathcal{O}}(Q\mathbf{X}) \ar[d] \\
\End_{\mathcal{O}}(\mathbf{X}) \ar[r] & \Hom_{\mathcal{O}}(Q\mathbf{X}, \mathbf{X})  }
$$
Here $\End_{\mathcal{O}}(\mathbf{X})$ denotes the restricted endomorphism $C$-colored operad associated to $\mathbf{X}$ as in~(\ref{endp}), 
and $\Hom_{\mathcal{O}}(\mathbf{X}, \mathbf{Y})$ is the $C$-colored collection defined by
$$
\Hom_{\mathcal{O}}(\mathbf{X},\mathbf{Y})(c_1,\ldots, c_n;c)
=
\Map(X(c_1)\wedge \cdots\wedge X(c_n), Y(c))\; \mbox{ if \;$\mathcal{O}(c_1,\ldots,c_n;c)\ne0$},
$$
and zero otherwise, as in (\ref{homp}).
Here $\tau$ is a trivial fibration of operads, since it is the pullback of a trivial fibration. 
Since the $C$-colored operad $\mathcal{O}$ is cofibrant there exists a lifting indicated in the diagram:
$$
\xymatrix{
& \End_{\mathcal{O}}(\varepsilon_\mathbf{X})\ar[d]^{\tau} \\
\mathcal{O} \ar[r] \ar@{.>}[ur] & \End_{\mathcal{O}}(\mathbf{X}) }
$$
That lifting endows $Q\mathbf{X}$ with an $\mathcal{O}$-algebra structure such that $\varepsilon_{\mathbf{X}}$ is a map of $\mathcal{O}$-algebras.

In order to prove uniqueness, 
suppose $\gamma, \gamma'\colon \mathcal{O}\to \End_{\mathcal{O}}(Q\mathbf{X})$ are $\mathcal{O}$-algebra structures on $Q\mathbf{X}$ and 
$\varepsilon_{\mathbf{X}}$ is a map of $\mathcal{O}$-algebras for each of them, 
i.e., 
$\gamma$ and $\gamma'$ factors through $\End_{\mathcal{O}}(\varepsilon_{\mathbf{X}})$. 
Let $\delta,\delta'\colon \mathcal{O}\to \End_{\mathcal{O}}(\varepsilon_{\mathbf{X}})$ be the maps such that $\gamma=\rho\circ\delta$ and 
$\gamma'=\rho\circ\delta'$. 
Since $\tau$ is a trivial fibration and $\tau\circ\delta=\tau\circ\delta'$, 
it follows that $\delta$ and $\delta'$ are equal in the homotopy category, 
and hence so are $\gamma$ and $\gamma'$.
\end{proof}

Next we state the analogue of Theorem~\ref{aug_alg} for coaugmented functors.
The proof of this result is basically the same as for augmented functors; 
we leave further details to the interested reader.
\begin{theorem}
Let $(R_1,\eta_1),\ldots, (R_n, \eta_n)$ be coaugmented functors on  a simplicial monoidal model category $\mathcal{M}$. 
Let $\mathcal{O}$ be a cofibrant $C$-colored operad in simplicial sets and consider the extension $(R,\varepsilon)$ of $R_1,\ldots,R_n$ over 
$\mathcal{M}^C$ relative to $J_1,\ldots, J_n\subseteq C$. 
Let $\mathbf{X}$ be an $\mathcal{O}$-algebra and suppose that the map
$$
\End_{\mathcal{O}}(R\mathbf{X})\longrightarrow \Hom_{\mathcal{O}}(\mathbf{X},R\mathbf{X})
$$
of $C$-colored collections induced by $\eta_{\mathbf{X}}$ is a trivial fibration. 
Then $R\mathbf{X}$ admits a (homotopy unique) $\mathcal{O}$\nobreakdash-algebra structure such that $\eta_{\mathbf{X}}$ is a map between
$\mathcal{O}$-algebras.
\label{coaug_alg}
\end{theorem}

\begin{remark}
Theorems \ref{aug_alg} and \ref{coaug_alg} hold true if the $C$-colored operad $\mathcal{O}$ is non-symmetric. 
We may replace $\mathcal{O}$ by its symmetric version $\Sigma \mathcal{O}$, 
and both yield the same class of algebras (see Remark~\ref{symmetrization}).
\end{remark}

\section{Slices of rings and modules} 
\label{hfgdfg}

Let $\caM$, $\caC$ and  $\{\caC_i\mid i\in \Z\}$ be as in Section~\ref{hygrf}. 
We assume that the family $\{\caC_i\mid i\in \Z\}$ satisfies the axioms (A1)--(A5).
Recall from Appendix~\ref{jnhgd} that $E_{\infty}$ is a $\Sigma$-cofibrant resolution of the commutative operad and $A_{\infty}$ is a 
$\Sigma$-cofibrant resolution of the associative operad. 
For an $E_\infty$- or $A_\infty$-object $\E$ in $\caM$, 
we denote by $\D(\E)=\Ho(\Mod(Q\E))$ the homotopy category of (left) $Q\E$-modules, 
where $Q\E \to \E$ is a cofibrant replacement in the category of $E_{\infty}$- or $A_{\infty}$-algebras 
(using the semi model structure reviewed in Appendix~\ref{jnhgd}).
Our first two results show that the full slice functor $s_*$ preserves ring and module structures.
\begin{theorem} \label{hgdfggg}
Let $\E$ be an $E_\infty$- resp.\ $A_\infty$-object in $\caM$.
Then the full slice $s_* \E$ has the natural structure of an $E_\infty$- resp.\ $A_\infty$-object in $\caM^\Z$.
\end{theorem}

\begin{theorem} \label{jyhtgdd}
Let $\E$ be an $E_\infty$- or $A_\infty$-object in $\caM$ and $\M$, $\sfN$ be $\E$-modules. 
\begin{itemize}
\item[{\rm (i)}] The object $s_*\M$ has a natural structure of an $s_* \E$-module in $\caM^\Z$ (in the strict sense). 
The assignment $\M \mapsto s_* \M$ can be enhanced to a functor $\D(\E) \to \D(s_* \E)$. 
This functor is exact provided every $\caC_i$ is a triangulated subcategory.
\item[{\rm (ii)}] If $\E$ is an $E_\infty$-object, 
then the functor $\D(\E) \to \D(s_*(\E))$ is lax symmetric monoidal.
In the motivic slice filtration, 
the natural transformation
\begin{equation}
\label{motivic_transf}
s_*(\M) \wedge_{s_*(\E)} s_*(\sfN) \longrightarrow s_*(\M \wedge_\E \sfN)
\end{equation}
is an isomorphism if $\M$ or $\sfN$ is $\E$-cellular.
\end{itemize}
\end{theorem}

The theorems can be applied in the case when $\E$ is the unit $\unit$ of $\caM$. 
In particular, 
this shows that the full slice functor $s_* \colon \caC \to \caC^\Z$ can be refined to a functor $\caC \to \Mod(s_*(\unit))$.

The statements of Theorems \ref{hgdfggg} and \ref{jyhtgdd} employ the axiomatic setups of Sections \ref{hygrf} and \ref{graded_slice}. 
However, for the proofs,
which take up a significant part of this section, 
it will be convenient to work in the setup of Section~\ref{jytgef}.
Recall the notations $\caN$, $\caD$, $\caD_0$ and $\caD_1$ from Section~\ref{jytgef} and assume the pair $(\caD_0,\caD_1)$ satisfies (B1)--(B2).  
The following contains some standard notions from Appendices \ref{jnhgd} and \ref{ytrgvv}.

\begin{proposition} \label{htgrfdd}
Let $\E$ be an $E_\infty$- resp.\ $A_\infty$-object in $\caN$.
Then the colocalization $c_0 \E$ has a (homotopy unique) natural structure of an $E_\infty$- resp.~$A_\infty$-object in $\caN$ such that $c_0\E\to\E$ 
is represented by a zig-zag of maps of $E_\infty$- resp.~$A_\infty$-objects.
\end{proposition}

\begin{proof}
We may assume the $E_\infty$- or $A_\infty$-operad is cofibrant 
(if not, then we can appeal to a zig-zag of corresponding algebra objects).
In order to apply Theorem~\ref{aug_alg} we first introduce necessary notation.
Let $\mathscr{K}$ be a set of objects of $\caD$ that generates $\caD_0$ by homotopy colimits.
Fix a cofibrant replacement functor $Q$ in the colocalization $C^{\mathscr{K}}\caN$ of $\caN$ with respect to $\mathscr{K}$ 
(see Definition~\ref{Bcoloc}).
There is a natural augmentation $\varepsilon\colon Q\to\mathrm{Id}$ and $Q$ models the functor $c_0$.
We may also assume that $\E$ is underlying fibrant by using the semi model structure on $E_\infty$- or $A_\infty$-algebras 
(in fact a zig-zag by first replacing $\E$ cofibrantly).
In order to verify the assumptions in Theorem \ref{aug_alg}, 
we show that 
\begin{equation}
\Map(Q\E \otimes \stackrel{(n)}{\cdots} \otimes Q\E,Q\E) 
\longrightarrow 
\Map(Q\E \otimes\stackrel{(n)}{\cdots} \otimes Q\E,\E)
\label{equ01}
\end{equation}
is a trivial fibration of simplicial sets for every $n\ge 0$. 

By assumption, 
$Q\E$ and the unit are cofibrant in $\caN$ and the map $Q\E \to \E$ is a fibration. 
This shows that (\ref{equ01}) is a fibration for every $n\ge 0$.
Moreover, $Q\E$ and $\E$ are fibrant so that the simplicial enrichments coincide with the homotopy function complexes.
Now $Q\E \otimes \cdots \otimes Q\E$ is $\mathscr{K}$-colocal since $\caD_0$ is closed under the tensor product and it contains the unit. 
Moreover, 
$Q\E \to \E$ is a $\mathscr{K}$-colocal equivalence.
This implies that (\ref{equ01}) is a weak equivalence.
\end{proof}

\begin{proposition} \label{jyssrf}
Let $\E$ be an $E_\infty$- resp.\ $A_\infty$-object in $\caN$ with underlying object in $\caD_0$.
Then the localization $l_1 \E$ has a (homotopy unique) natural structure of an $E_\infty$- resp.\ $A_\infty$-object in $\caN$ such that $\E\to l_1\E$ 
is represented by a zig-zag of maps of $E_\infty$- resp.\ $A_\infty$-objects.
\end{proposition}
\begin{proof}
Again, we may assume the $E_\infty$- or $A_\infty$-operad is cofibrant. 
In order to apply Theorem \ref{coaug_alg} we let $\mathscr{K}$ be a set of objects of $\caD$ that generates $\caD_1$ by homotopy colimits.
Form the set of maps $\mathscr{S}=\{0 \to K \mid K \in \mathscr{K} \}$ and let $R$ be a fibrant replacement functor in the corresponding Bousfield 
localization $L_\mathscr{S}\caN$ of $\caN$ (see Definition~\ref{Bloc}). 
There is a coaugmentation $\eta \colon \mathrm{Id} \to R$ and $R$ models the functor $l_1$. 
We may also assume that $\E$ is underlying cofibrant by using the semi model structure on $E_\infty$- or $A_\infty$-algebras. 
In order to verify the assumptions in Theorem \ref{coaug_alg}, 
we show that 
\begin{equation}
\Map(R\E \otimes \stackrel{(n)}{\cdots} \otimes R\E,R\E) \longrightarrow \Map(\E \otimes \stackrel{(n)}{\cdots} \otimes \E, R\E)
\label{equ02}
\end{equation}
is a trivial fibration of simplicial sets for every $n\ge 0$. 

By assumption, 
$R\E$ is fibrant in $\caN$, $\E$ is cofibrant and $\E \to R \E$ is a cofibration. 
This shows that (\ref{equ02}) is a fibration for every $n\ge 0$.
Moreover, 
$R \E$ is also cofibrant. 
Thus the simplicial enrichments coincide with the homotopy function complexes.
To see that (\ref{equ02}) is a weak equivalence we show that $\E \otimes \cdots \otimes \E \to R\E \otimes \cdots \otimes R\E$
is an $\mathscr{S}$-equivalence. 
Lemma \ref{jyhgd} implies that $R\E$ is in $\caD_0$.
Applying Lemma \ref{bfgdhs} repeatedly finishes the proof.
\end{proof}

\begin{proof}[Proof of Theorem \ref{hgdfggg}:]
Let $\caM$, $\caC$ and  $\{\caC_i\mid i\in \Z\}$ be as in Section~\ref{hygrf}, 
and let $\caN=\caM^\Z$, $\caD_0=\prod_{i \in \Z} \caC_i$ and $\caD_1=\prod_{i \in \Z} \caC_{i+1}$, 
as in the second part of Section~\ref{graded_slice}.
We denote the corresponding colocalization and localization functors for $\caD_0$ and $\caD_1$ by $c_0'$, $c_1'$, $l_0'$ and $l_1'$. 
The diagonal functor $d \colon \caM \to \caM^\Z = \caN$ is right adjoint to the sum functor $\caN \to \caM$, 
which is a symmetric monoidal left Quillen functor, 
hence it is lax symmetric monoidal.
Thus $d$ preserves $E_\infty$- and $A_\infty$-objects.

Suppose $\E$ is an $E_\infty$- or $A_\infty$-object in $\caM$. 
By Proposition \ref{htgrfdd}, 
$c_0'(d(\E))$ can be modeled as an $E_\infty$- resp.~$A_\infty$-object. 
Applying Proposition \ref{jyssrf} shows that $s_*(\E)=l_1'(c_0'(d(\E)))$ can be modeled as an $E_\infty$- resp.~$A_\infty$-object.
\end{proof}

Let $\caN$, $\caD$ and $\caD_0$ be as in Section~\ref{jytgef}. 
Suppose $\caD_0$ satisfies (B1) and $\E$ is an $A_\infty$-object in $\caN$ (the $E_\infty$-case can be treated similarly).
We assume that $\E\in\caD_0$ and let $\caN'=\Mod(\E)$ be the category of left $\E$-modules.
By \cite[Corollary 2.3.8.(1)]{lurie.DAGII} the category $\Mod(\E)$ is locally presentable. 
Assume also that $\E$ is cofibrant.
Then $\Mod(\E)$ is a cofibrantly generated model category with the transferred model structure.
Denote the corresponding derived category by $\caD'=\D(\E)=\Ho(\caN')$.
Fix a set $\mathscr{K}$ of objects of $\caD_0$ that generates $\caD_0$ by homotopy colimits.
The free modules $\mathscr{K}'=\E \otimes \mathscr{K}$ on $\mathscr{K}$ generate a full subcategory $\caD_0'$ of $\caD'$ by homotopy colimits. 
Denote by $U \colon \caD' \to \caD$ the forgetful functor.

\begin{lemma} \label{bgfgbfg}
Let $X$ be an object of $\caD'$. 
Then $U(X)\in\caD_0$ if and only if $X\in\caD_0'$.
\end{lemma}
\begin{proof}
Note that $U(\mathscr{K}') \subset \caD_0$ since $\E$ lies in $\caD_0$ which is closed under the tensor product.
This shows the ``if'' part of the equivalence because $\caD_0'$ is generated by homotopy colimits and $U$ commutes with homotopy colimits.

For the ``only if'' part of the equivalence we shall employ a trick with homotopy colimits.
Suppose $X$ is an object in $\caD'$ such that $U(X)\in\caD_0$ and write 
$$
X 
\simeq 
\hocolim_{[n] \in \bigtriangleup^\mathrm{op}} (\E^{\otimes (n+1)} \otimes U(X)).
$$
Every term $\E^{\otimes (n+1)} \otimes U(X)$ belongs to $\caD_0'$ which is closed under homotopy colimits, 
so we may conclude $X\in\caD_0'$.
\end{proof}

\begin{corollary}
The category $\caD_0'$ is closed under extensions in $\caD'$.
\end{corollary}
\begin{proof}
This follows from Lemma \ref{bgfgbfg}.
\end{proof}

\begin{proposition} \label{hgfdsdd}
For every $X$ in $\caD'$ the map $U(c_0' X) \to U(X)$ is a $\mathscr{K}$-colocalization, where
$c_0'$ denotes the right adjoint of the inclusion $\caD_0' \to \caD'$.
\end{proposition}
\begin{proof}
The object $U(c_0'X)$ is $\mathscr{K}$-colocal by Lemma \ref{bgfgbfg} and $U$ sends $\mathscr{K}'$-colocal equivalences to $\mathscr{K}$-colocal 
equivalences by adjointness.
\end{proof}

\begin{proposition} \label{hhgfddd}
Suppose $\E$ is an $E_\infty$- or $A_\infty$-object in $\caN$ and $\M$ is a left $\E$-module. 
Then $c_0 \M$ has naturally the structure of a $c_0 \E$-module and $c_0 \M \to \M$ is a map over $c_0 \E \to \E$. 
The assignment $\M \mapsto c_0 \M$ defines a functor $\D(\E) \to \D(c_0 \E)$.
\end{proposition}
\begin{proof}
First apply the pullback functor $\D(\E)\to\D(c_0\E)$. 
The required colocalization in $\D(c_0 \E)$ follows from Proposition \ref{hgfdsdd}.
\end{proof}

With reference to the setup with $\caN$, $\caD$ and $\caD_1\subset\caD_0$ of Section~\ref{jytgef} we assume the axioms (B1)--(B2) hold. 
In the following we keep the notations $\caN'$, $\caD'$, $\mathscr{K}$, $\mathscr{K}'$ and $\caD_0'$ from above.
Let $\E$ be an $E_\infty$- or $A_\infty$-object in $\caN$ which lies in $\caD_0$. 
We choose a set $\mathscr{K}_1$ that generates $\caD_1$ by homotopy colimits and define $\mathscr{K}_1'= \E \otimes \mathscr{K}_1$. 
Then $\mathscr{K}_1'$ generates a full subcategory $\caD_1'$ of $\caD'$ by homotopy colimits. 
Proposition \ref{hgfdsdd} shows that the associated colocalization and localization functors $c_0', c_1'$, $l_0'$ and $l_1'$ lie above the 
corresponding functors $c_0, c_1$, $l_0$, $l_1$.
Since we will be working with several ring objects we introduce the notations $\D(\E)_0=\caD_0'$ and $\D(\E)_1=\caD_1'$. 
Let $\caQ_\E$ be the core of the pair $(\D(\E)_0,\D(\E)_1)$. 
Finally, 
we set $l_1^\E=l_1'$ (we use the same notations for non-cofibrant $\E$ by an implicit cofibrant replacement).

\begin{lemma}
Suppose $\sfA \to \sfB$ is a map of $E_\infty$- or $A_\infty$-objects in $\caN$ and $\sfA$, $\sfB\in \caD_0$.
Then the adjunction given by extension and restriction
$$
\xymatrix{
\D(\sfA) \ar@<4pt>[r] & \ar@<1pt>[l] \D(\sfB)
}
$$
induces adjunctions:
$$
\xymatrix{
\D(\sfA)_0 \ar@<4pt>[r] & \ar@<1pt>[l] \D(\sfB)_0
}
\mbox{ and }
\xymatrix{
\caQ_{\sfA} \ar@<4pt>[r] & \ar@<1pt>[l] \caQ_{\sfB}
}
$$
\end{lemma}
\begin{proof}
The first part follows from Lemma \ref{bgfgbfg}.
For the second part we consider the composite adjunction:
$$
\xymatrix{
\D(\sfA)_0 \ar@<4pt>[r] & \ar@<1pt>[l] \D(\sfB)_0 \ar@<4pt>[r] & \ar@<1pt>[l] \caQ_{\sfB}
}
$$
The right adjoint of this adjunction factors through $\caQ_{\sfA} \subset \D(\sfA)_0$ because the local objects in the module
categories are the underlying local objects. 
This finishes the proof.
\end{proof}

\begin{lemma} \label{kkyhgf}
Let $\sfA \to \sfB$ be a map of $E_\infty$- or $A_\infty$-objects in $\caN$ which is an $l_1$\nobreakdash-equivalence. 
Suppose $\sfA$ and $\sfB$ are in $\caD_0$.
Then the induced functor $\caQ_{\sfA} \to \caQ_{\sfB}$ is an equivalence.
\end{lemma}
\begin{proof}
To prove that the unit and counit of the adjunction 
$$
\xymatrix{
\caQ_{\sfA} \ar@<4pt>[r] & \ar@<1pt>[l] \caQ_{\sfB}
}
$$
are isomorphisms it suffices to show there is a 
naturally induced $l_1$-equivalence $\M\to\M\otimes_{\sfA}\sfB$ for every $\M \in \D(\sfA)_0$.
Recall that $\caD_0 \otimes \sfA$ generates $\D(\sfA)_0$ by homotopy colimits.
Thus we may assume $\M=\M'\otimes\sfA$, where $\M' \in \caD_0$.
In this case the map in question is $\M'\otimes\sfA\to\M'\otimes\sfB$,  
which is an $l_1$-equivalence by Lemma \ref{bfgdhs}.
\end{proof}

\begin{proposition} \label{dddghhh}
Suppose $\E$ is an $E_\infty$- or $A_\infty$-object in $\caN$ which lies in $\caD_0$.
Then the localization functor $l_1^\E \colon \D(\E)_0 \to \D(\E)_0$ factors through $\D(l_1\E)_0$. 
More precisely, 
there is a $2$-commutative diagram:
$$
\xymatrix{
\D(\E)_0 \ar[r]^{l_1^E} \ar@{.>}[rd] & \D(\E)_0 \\
 & \D(l_1\E)_0 \ar[u]_U }
$$
\end{proposition}
\begin{proof}
There is an evident $2$-commutative diagram:
$$
\xymatrix{
\caQ_\E \ar@{^{(}->}[r] & \D(\E)_0 \\
\caQ_{l_1\E} \ar[u] \ar@{^{(}->}[r] & \D(l_1\E)_0 \ar[u]_U }
$$
The left vertical map is an equivalence by Lemma \ref{kkyhgf}.
Hence the diagram
$$
\xymatrix{
\D(\E)_0 \ar[r] & \caQ_\E \ar@{^{(}->}[r] \ar[d]^\simeq & \D(\E)_0 \\
& \caQ_{l_1\E} \ar@{^{(}->}[r] & \D(l_1\E)_0 \ar[u]_U }
$$
commutes as well, 
and the upper horizontal composite map is $l_1^\E$. 
This shows the claim.
\end{proof}

\begin{proof}[Proof of Theorem \ref{jyhtgdd}\,{\rm (i)}:] 
Let $\E$ be an $E_\infty$- or $A_\infty$-object in $\caM$.
We adopt the same notation as in the proof of Theorem \ref{hgdfggg}. 
There is an induced functor $\D(\E) \to \D(d(\E))$.
Proposition \ref{hhgfddd} implies the $c_0'$-colocalization can be modeled on the level of $\E$-modules by a functor
$$
c_0^{d(\E)} \colon \D(d(\E)) \longrightarrow \D(c_0' d(\E))_0.
$$
By Proposition \ref{dddghhh}, 
the $l_1'$-localization acquires a factorization on the level of modules 
$$
l_1^{c_0' d(\E)} \colon  \D(c_0' d(\E))_0 \longrightarrow \D(l_1' c_0' d(\E))_0.
$$
Since $l_1' c_0' d(\E) = s_* \E$ the composite of $c_0^{d(\E)}$ and $l_1^{c_0' d(\E)}$ yields the desired functor
$$
\D(\E) \longrightarrow \D(s_* \E).
$$
This proves the second part of the claim.
In particular, 
for an $\E$-module $\M$ the full slice $s_* \M$ has the structure of an $s_*\E$-module,
showing also the first part of the claim.
\end{proof}

Our proof of the second part of Theorem~\ref{jyhtgdd} employs symmetric monoidal structures on module categories. 
It will be convenient to consider a particular $E_\infty$ operad, 
namely the linear isometries operad \cite[\S I.3]{EKMM}.
From now on $E_\infty$ denotes the image of the linear isometries operad in simplicial sets. 
If $\E$ is an  $E_\infty$-algebra then $\Mod(\E)$ is a symmetric monoidal model category with a weak unit,
see \cite[\S 9]{Spi}. 
In particular, 
$\D(\E)$ is a closed symmetric monoidal category.
Recall the standing assumptions: 
$\D(\E)_0 \subset \D(\E)$ is a symmetric monoidal subcategory, 
and $\D(\E)_1 \subset \D(\E)_0$ is a tensor ideal.

\begin{lemma}
The core $\caQ_\E$ has an induced symmetric monoidal structure and $\D(\E)_0 \to \caQ_\E$ is a symmetric monoidal functor.
\end{lemma}
\begin{proof}
This follows from the $\E$-analogue of Lemma \ref{bfgdhs}.
\end{proof}

\begin{proof}[Proof of Theorem \ref{jyhtgdd}\,{\rm (ii)}:]
The claim for $\D(\E)\longrightarrow\D(s_* \E)$ follows because $c_0^{d(\E)}$ and $l_1^{c_0' d(\E)}$ are 
lax symmetric monoidal functors, cf.~the proof of Theorem \ref{jyhtgdd}(i).

The second part can be proven as follows:
For an $\E$-module $\M$ the natural transformation (\ref{motivic_transf}) is an isomorphism if $\sfN$ is a motivic $\E$-sphere. 
By exactness of both functors involved this extends to the localizing subcategory generated by motivic $\E$-spheres.
\end{proof}

Our next objective is to show that the colocalization functors preserve rings and modules over effective rings, 
and similarly that the localization functors preserve effective rings and modules over effective rings. 
In what follows we employ $\Sigma$-cofibrant resolutions of the colored operads $\LMod_{\mathcal{O}}$ for left $\mathcal{O}$-modules 
and $\Mor_{\mathcal{O}}$ for maps of $\mathcal{O}$-algebras;  
we refer to Subsection~\ref{examples_operads} for further details.
Throughout we adopt the setup of Section~\ref{hygrf}.

\begin{theorem} \label{jhdsdd}
Suppose $\E$ is an $E_\infty$- resp.\ $A_\infty$-object in $\caM$.
Then $c_0(\E)$ has a (homotopy unique) $E_\infty$- resp.\ $A_\infty$-structure such that $c_0 \E \to \E$ is a map of 
$E_\infty$-\nobreakdash resp.\ $A_\infty$-objects.
\end{theorem}
\begin{proof}
Apply Proposition \ref{htgrfdd} to $\caN=\caM$, $\caD=\caC$ and $\caD_0 = \caC_0$.
\end{proof}

\begin{theorem}
Suppose $\E$ is an effective $E_\infty$- resp.\ $A_\infty$-object in $\caM$ and $\M$ is an $E_{\infty}$- resp.\ $A_{\infty}$-module over $\E$. 
Then $c_q \M$ has a (homotopy unique) $E_{\infty}$- resp.\ $A_{\infty}$-module structure over $\E$ such that $c_q \M \to \M$ is a map of 
$E_{\infty}$- resp.\ $A_{\infty}$-modules. 
Moreover, for $q' \ge q$,  
$c_{q'} \M \to c_q \M$ is an $E_{\infty}$- resp.\ $A_{\infty}$-map.
\label{thm:coloc_modules}
\end{theorem}
\begin{proof}
We may assume that $(\LMod_{\Com})_\infty$ and $(\LMod_{\Ass})_\infty$ are cofibrant 
(if not, then we can appeal to a zig-zag of corresponding algebra objects).
Fix a set $\mathscr{K}_q$ of objects of $\caC$ that generates $\caC_q$ by homotopy colimits.
Let $Q_q$ be a cofibrant replacement functor in  the Bousfield colocalization $C^{\mathscr{K}_q} \caM$ of $\caM$ with respect to $\mathscr{K}_q$.
There is a natural augmentation $\varepsilon_q\colon Q_q\to\mathrm{Id}$ and $Q_q$ models the functor $c_q$. 
For our purposes it suffices to verify the assumptions in Theorem~\ref{aug_alg} for $n=1$ and $J_1=\{m\}$.
We may assume $(\E,\M)$ is cofibrant and underlying fibrant using semi model structures (if not, we may appeal to a zig-zag). 
It remains to show there is a trivial fibration of simplicial set:
\begin{equation}
\label{equ03}
\xymatrix{\Map(\E \otimes \cdots \otimes \E \otimes Q_q \M \otimes \E \otimes
\cdots \otimes \E,Q_q \M) \ar[d] \\
\Map(\E \otimes \cdots \otimes \E \otimes Q_q \M \otimes \E \otimes
\cdots \otimes \E,\M)}
\end{equation}
By assumption, 
$\E$ and $Q_q \M$ are cofibrant and the map $Q_q \M \to \M$ is a fibration;
this proves (\ref{equ03}) is a fibration.
Moreover, $Q_q \M$ and $\M$ are fibrant so the simplicial enrichments coincide with the homotopy function complexes.
The claim follows since
$$
\E \otimes \cdots \otimes \E \otimes Q_q \M \otimes \E \otimes \cdots \otimes \E
$$
is $\mathscr{K}_q$-colocal by axiom (A5) since $\E\in\caC_0$ and $Q_q \M\in\caC_q$.

In the proof of the second claim, 
let $(\Mor_{\mathcal{O}})_{\infty}$ be a cofibrant resolution of the $4$-colored operad $\Mor_{\mathcal{O}}$ for $\mathcal{O}=\LMod_{\Ass}$ or 
$\mathcal{O}=\LMod_{\Com}$, 
cf.~Subsection~\ref{inftymaps}.
Now let $n=2$, $J_1=(m,0)$, $J_2=(m,1)$, $Q_1=Q_{q'}$ and $Q_2=Q_q$ in the formulation of Theorem \ref{aug_alg}.
Applying the extension $Q$ to the $(\Mor_{\mathcal{O}})_{\infty}$-algebra determined by the identity map on $(\E,\M)$ finishes the proof.
\end{proof}

\begin{theorem} \label{juyfff}
Suppose $\E$ is an $E_\infty$- resp.\ $A_\infty$-object in $\caM$ and $\M$ is an $E_\infty$- resp.\ $A_\infty$-module over $\E$.
Then for every $q\in\Z$ the object $c_q \M$ has a (homotopy unique) $E_{\infty}$- resp.\ $A_{\infty}$-module structure over $c_0 \E$ such that 
$c_q \M \to \M$ is a map of $E_{\infty}$- resp.\ $A_{\infty}$-modules. 
Moreover, 
for $q' \ge q$, 
$c_{q'} \M \to c_q \M$ is an $E_{\infty}$- resp.\ $A_{\infty}$-map over $c_0\E$.
\end{theorem}
\begin{proof}
This follows from Theorem \ref{thm:coloc_modules} by first restricting to $c_0 \E$.
\end{proof}

\begin{theorem}
Suppose $\E$ is an effective $E_\infty$- resp.\ $A_\infty$-object in $\caM$ and $i\ge 1$.
Then $l_i\E$ has a (homotopy unique) $E_\infty$- resp.\ $A_\infty$-structure such that $\E\to l_i\E$ is a map of $E_\infty$- resp.\ $A_\infty$-objects.
\label{thm:loc_rings}
\end{theorem}
\begin{proof}
Apply Proposition \ref{jyssrf} to $\caN=\caM$, $\caD=\caC$, $\caD_0 = \caC_0$ and  $\caD_1=\caC_i$.
\end{proof}

\begin{theorem}
Suppose $\E$ is an effective $E_\infty$- resp.\ $A_\infty$-object in $\caM$, 
$\M$ is an $E_\infty$- resp.\ $A_\infty$-module over $\E$ and $\M\in\caC_q$. 
Then, 
for $i\ge 1$ and $q\in\Z$, 
$l_{q+i}\M$ has a (homotopy unique) $E_\infty$- resp.\ $A_\infty$-module structure over $l_i \E$ such that $\M\to l_{q+i}\M$ is an 
$E_\infty$- resp.\ $A_\infty$-module map over $\E\to l_i\E$.
\label{thm:loc_modules}
\end{theorem}
\begin{proof}
The proof is similar to the proof of Proposition \ref{jyssrf} by using the cofibrant operads $(\LMod_{\Ass})_\infty$ and $(\LMod_{\Com})_\infty$.
In this case we apply Theorem \ref{coaug_alg} to $n=2$, 
$J_1=\{r\}$, 
$J_2=\{m\}$, 
$R_1$ a fibrant replacement functor on the model category level modelling $l_i$ and $R_2$ a model for $l_{q+i}$. 
A main step in the proof is to show that
$$
\E \otimes \cdots \otimes \E \otimes \M \otimes \E \otimes \cdots \otimes \E
\longrightarrow 
l_i \E \otimes \cdots \otimes l_i \E \otimes l_{q+i}\M \otimes l_i \E \otimes \cdots \otimes l_i \E
$$
is an $l_{q+i}$-equivalence. 
This follows by applying Lemma \ref{ngfsd} repeatedly.
\end{proof}

\begin{corollary} \label{dgdgggg}
Suppose $\E$ is an $E_\infty$- resp.\ $A_\infty$-object in $\caM$ and $\M$ is an $E_\infty$- resp. $A_\infty$\nobreakdash-module over $\E$.
Then, 
for every $q\in\Z$, 
$s_q \M$ has a (homotopy unique) $A_\infty$- resp. $E_\infty$-module structure over $s_0 \E$ such that $c_q \M\longrightarrow s_q \M$ is an 
$E_\infty$- resp. $A_\infty$-module map over $c_0 A \to s_0 A$.
\end{corollary}
\begin{proof}
This follows from Theorems \ref{juyfff} and \ref{thm:loc_modules} for $i=1$.
\end{proof}

\begin{theorem} \label{eeerhgh}
Suppose $\E_1 \to \E_2$ is a map between $E_\infty$- resp.\ $A_\infty$-objects in $\caM$. 
Then $c_0 \E_1 \to c_0 \E_2$ is a map between $E_\infty$- resp.\ $A_\infty$-objects. 
Moreover, 
for $i\ge 1$, 
$l_i c_0 \E_1 \to l_i c_0 \E_2$ is a map between $E_\infty$- resp.\ $A_\infty$-objects.
\end{theorem}
\begin{proof}
The argument is a verbatim copy of the proofs of Propositions \ref{htgrfdd} and \ref{jyssrf}. 
Take cofibrant resolutions of the operads $\Mor_{\Com}$ and $\Mor_{\Ass}$ and proceed by colocalizing and localizing at the colors.
\end{proof}

\section{Motivic spectra}
\label{asgfcdj}

In this section we specialize to motivic stable homotopy theory.
First we recall some pertinent parts of the motivic setup introduced in Section \ref{jjhhhh}. 

The functor $f_q$ is the cofibrant replacement functor of the colocalized model structure of $\SH(S)$ with respect to 
$$
\scK(q)=\{\Sigma^{0,q}\Sigma^\infty_+ X[n] \mid X\in \Sm/S \mbox{ and }n,q \in \Z\}.
$$ 
The functor $l_q$ is the fibrant replacement functor of the localized model structure of $\SH(S)$ with respect to 
$$
\scS(q)=\{0\to \Sigma^{0,q}K \mid K\in \scK=\scK(0)\}.
$$
The $q$th slice functor $s_q$ is obtained by first colocalizing with respect to $\scK(q)$ and second localizing with respect to $\scS(q+1)$.
Our results in Section \ref{hfgdfg} specialize to the following assertions (see also the other examples in Section~\ref{jjhhhh}):
\begin{itemize}
\item[{\rm (i)}] 
The full slice $s_* \E$ of any motivic $E_\infty$-ring spectrum $\E$ is a motivic $E_\infty$-ring spectrum in $\Spt_T^\Sigma(S)^\Z$.
\item[{\rm (ii)}] 
The full slice functor $s_*$ lifts to an exact functor between module categories
$$
\D(\E) \longrightarrow \D(s_*\E).
$$
In particular, 
there exists an enhancement $\SH(S)\to\D(s_* \unit)$ of the full slice functor.
\item[{\rm (iii)}] 
The natural transformation
$$
s_*(\M) \wedge_{s_*(\E)} s_*(\sfN) \longrightarrow s_*(\M \wedge_\E \sfN)
$$
in~(\ref{motivic_transf}) is an isomorphism if $\M$ or $\sfN$ is $\E$-cellular.
\item[{\rm (iv)}] 
If $\E$ is a motivic $E_\infty$-ring spectrum, 
then $f_0 \E$ and $s_0 \E$ are motivic $E_\infty$-ring spectra.
\item[{\rm (v)}] 
If $\M$ is a module over a motivic $E_\infty$-ring spectrum $\E$, 
its $(q-1)$-connective cover $f_q \M$ is an $f_0 \E$-module and the slice $s_q \M$ is an $s_0 \E$-module.
\end{itemize}

\subsection{Motivic sphere and Eilenberg-Mac\,Lane spectra}
Let $\MZ$ be the integral motivic Eilenberg-Mac\,Lane spectrum.
If the base scheme is a perfect field, 
then $s_0\MZ=\MZ$, $s_i\MZ=0$ for $i\neq 0$.
Similarly, 
$s_0\mathsf{M}R=\mathsf{M}R$, $s_i\mathsf{M}R=0$ for $i\neq 0$, $R$ a commutative ring.
These computations follow since $s_0\unit=\MZ$ by \cite{Levine-slices}, \cite{Voevodsky-zeroslices}.
By applying $s_0$ to the unit map $\unit\to\MZ$ it follows that the $E_\infty$-structure on $\MZ$ obtained by our method coincides with the usual one. 
This also shows uniqueness of the $E_\infty$-structure on $\MZ$.

Over fields of characteristic zero, 
the full slice $s_*\unit$ can be computed using the Morel-Hopkins isomorphism identifying $\MZ$ with a regular quotient of $\MGL$:
$s_i\unit$ is (up to a shift) the motivic Eilenberg-Mac\,Lane spectrum on the complex of abelian groups which computes the $E_2$-term of the 
topological Adams-Novikov spectral sequence with internal degree $i$; 
see \cite{voevodsky.open} for further details.
(The forthcoming paper \cite{roendigs-spitzweck-oestvaer.slices-sphere} deals with the slices of the motivic sphere spectrum.)
It is an open problem whether the $E_\infty$-structure on $s_* \unit$ can be described in purely topological terms.

Over fields of characteristic zero, 
the spectrum $\mathsf{M}R$ is cellular \cite[Remark 6.6]{Dugger-Isaksen}. 
Therefore (\ref{motivic_transf}) specializes to the isomorphisms
$$
s_*(\mathsf{M}\mathbb{Z}/p \wedge \mathsf{M}\mathbb{Z}/p) \cong
\mathsf{M}\mathbb{Z}/p \wedge_{s_* \unit} \mathsf{M}\mathbb{Z}/p
$$
and
$$
s_*(\MZ \wedge \MZ) \cong \MZ \wedge_{s_* \unit} \MZ.
$$
These isomorphisms express dual rigid Steenrod algebras in terms of derived tensor products over the full slice $s_*\unit$ of the motivic sphere 
spectrum.

\subsection{Algebraic $K$-theory}
Let $\KGL$ be the motivic $K$-theory spectrum. 
Over perfect fields, 
$s_i \KGL = \Sigma^{2i,i} \MZ$ for $i\in\Z$ by \cite{Levine-slices}, \cite{Voevodsky-zeroslices}.
Thus $s_* \KGL = s_0\KGL[\beta^{\pm 1}]$, 
$\vert\beta\vert=(2,1)$, 
produces a natural $E_\infty$-structure on the graded object 
\begin{equation}
\label{equation:knckju}
\bigvee_{i \in \Z} \Sigma^{2i,i} \MZ.
\end{equation}
For an alternate way of producing an $E_\infty$-structure on (\ref{equation:knckju}) we refer to \cite{spitzweck.periodizable}. 
It is not known whether these graded $E_\infty$-objects are equivalent.

For the connective cover $\kgl=f_0\KGL$ of algebraic $K$-theory there is a naturally induced map of $E_\infty$-ring spectra 
(the identification $s_0 \kgl = \MZ$ requires a perfect base field)
$$
\kgl 
\longrightarrow 
s_0 \kgl = \MZ.
$$
Motivic obstruction theory can be used to show that $\kgl$ has a unique $E_\infty$-structure \cite{NSOKtheory}.

\subsection{Hermitian $K$-theory}
Let $\KQ$ denote the hermitian $K$-theory spectrum. 
Over perfect fields of characteristic different from two, 
one can show \cite{R-O}
$$
s_q \KQ 
= 
\left\{
\begin{array}{ll}
\Sigma^{2q,q} \MZ \vee \bigvee_{i<0} \Sigma^{2q+2i,q} \mathsf{M}\mathbb{Z}/2 & q\equiv 0 (2), \\
\bigvee_{i<0} \Sigma^{2q+2i+1,q} \mathsf{M}\mathbb{Z}/2 & q\equiv 1 (2).
\end{array}
\right.
$$
In the forthcoming paper \cite{spitzweck.kq} it is shown that $\KQ$ admits a (unique) $E_\infty$-structure if the ring of functions of the base scheme 
contains $1/2$ and $\sqrt{-1}$. 
One may ask whether the $E_\infty$\nobreakdash-structure on $s_0\KQ=\MZ\vee\bigvee_{i<0}\Sigma^{2i,0}\mathsf{M}\mathbb{Z}/2$ can be obtained by base 
change from the evident $E_\infty$-structure on the topological spectrum $\mathsf{H}\mathbb{Z} \vee \bigvee_{i<0}\Sigma^{2i}\mathsf{H}\mathbb{Z}/2$. 
Even the homotopy ring structure on $s_0\KQ$ is unknown at present.

Over perfect fields of characteristic different from two, 
the $q$th slice of the Witt-theory spectrum 
$$
s_q \mathsf{KT}= \bigvee_{i \in \Z} \Sigma^{q+2i,q} \mathsf{M}\mathbb{Z}/2
$$
is computed in \cite{R-O}.
By Bott periodicity for Witt-theory,
$$s_* \mathsf{KT}= s_0 \mathsf{KT}[\beta^{\pm 1}], 
\vert\beta\vert=(1,1).
$$

Since $\mathsf{KT}$ is obtained from $\KQ$ by inverting the motivic Hopf element $\eta$ the $E_\infty$\nobreakdash-structure on $\KQ$ induces one on 
$\mathsf{KT}$. 
(This follows by applying the techniques developed in \cite{spitzweck.kq}.) 
It is an open question whether the $E_\infty$-structure on $s_0 \mathsf{KT}= \bigvee_{i \in \Z} \Sigma^{2i}\mathsf{M}\mathbb{Z}/2$ is the obvious one.

\subsection{Algebraic cobordism}
Voevodsky \cite{voevodsky.open} has conjectured that the slices of the algebraic cobordism spectrum $\MGL$ are given by 
\begin{equation}
\label{equation:sgfdgd}
s_q \MGL 
= 
\Sigma^{2q,q} \MZ \otimes \MU_{2q}.
\end{equation}
Here $\MU$ denotes the topological cobordism spectrum.
Over fields of characteristic zero, 
this conjecture has been verified in \cite{spitzweck.rel} assuming the Hopkins-Morel isomorphism.

As homotopy algebras, 
it is expected that 
$$
s_* \MGL=\MZ[x_1,x_2,x_3,\ldots], 
$$ 
where $\vert x_i\vert=(2i,i)$.
It is an open question to determine the graded $E_\infty$-structure on $s_* \MGL$.

\appendix
\section{Colored operads and algebras} \label{jnhgd}
Colored operads originated in the work of Boardman-Vogt \cite{BV} on homotopy invariant algebraic structures on topological spaces.
For more recent accounts we refer to \cite{BM07}, \cite{CGMV}, \cite{Elmendorf-Mandell} and \cite{Leinster}. 
One may view colored operads as convenient generalizations of ordinary operads and categories.
In this appendix we recall the definitions and basic properties of colored operads and algebras that are used in the main body of the paper. 
Throughout, 
$\mathcal{V}$ is a cocomplete closed symmetric monoidal category with tensor product $\otimes$, 
unit $I$, internal hom $\Hom_{\mathcal{V}}(-,-)$ and initial object $0$.
Let $C$ be a set, 
whose elements are referred to as \emph{colors}.

\begin{definition}
A \emph{$C$\nobreakdash-colored collection} $\mathcal{K}$ in $\mathcal{V}$ consists of a set of objects 
$\mathcal{K}(c_1,\ldots, c_n;c)$ in $\mathcal{V}$ for each $(n+1)$-tuple of colors $(c_1,\ldots, c_n,c)$, 
equipped with a right action of the symmetric group $\Sigma_n$ given by maps
$$
\alpha^*\colon \mathcal{K}(c_1,\ldots, c_n;c)\longrightarrow
\mathcal{K}(c_{\alpha(1)},\ldots, c_{\alpha(n)};c),
$$
where $\alpha\in\Sigma_n$. (If $n=0$ or $n=1$, then $\Sigma_n$ is the trivial group by convention.)
\end{definition}

A \emph{map} of $C$\nobreakdash-colored collections $\varphi\colon\mathcal{K}\to\mathcal{L}$ is a family of maps in $\mathcal{V}$
$$
\varphi_{c_1,\ldots,c_n;c}\colon \mathcal{K}(c_1,\ldots,c_n;c)\longrightarrow
\mathcal{L}(c_1,\ldots, c_n; c)
$$
for every $n\ge 0$ and $(n+1)$-tuples $(c_1,\ldots, c_n,c)$, 
that is compatible with the symmetric group actions. 
We let $\Coll_C(\mathcal{V})$ denote the category of $C$\nobreakdash-colored collections in $\mathcal{V}$. 

\begin{definition}
A \emph{$C$\nobreakdash-colored operad} $\mathcal{O}$ in $\mathcal{V}$ comprises a $C$\nobreakdash-colored collection with unit maps 
$I\to\mathcal{O}(c;c)$ for all $c\in C$, 
and for every $(n+1)$-tuple of colors $(c_1,\ldots, c_n,c)$ and tuples
$$
(a_{1,1},\ldots, a_{1,k_1}; c_1),\ldots, (a_{n,1},\ldots, a_{n,k_n};
c_n),
$$
a \emph{composition product} map
$$
\xymatrix{
\mathcal{O}(c_1,\ldots, c_n;c)\otimes \mathcal{O}(a_{1,1},\ldots,
a_{1,k_1};c_1)\otimes\cdots\otimes \mathcal{O}(a_{n,1},\ldots,
a_{n,k_n};c_n)\ar[d]
\\ \mathcal{O}(a_{1,1},\ldots,a_{1,k_1},a_{2,1},\ldots,a_{2,k_2},\ldots,a_{n,1},\ldots,a_{n,k_n};c) }
$$
that is compatible with the symmetric group actions and subject to the usual associativity and unitary compatibility relations; 
see, for example, \cite[\S 2]{Elmendorf-Mandell}.
\end{definition}

A \emph{map of $C$\nobreakdash-colored operads} is a map between the underlying $C$\nobreakdash-colored collections that is compatible with the 
unit maps and the composition product maps.
We denote the category of $C$\nobreakdash-colored operads in $\mathcal{V}$ by ${\Oper}_C(\mathcal{V})$. 
There is a free-forgetful adjunction:
$$
\xymatrix{
F \colon \Coll_C(\mathcal{V}) \ar@<4pt>[r] & \ar@<1pt>[l]\Oper_C(\mathcal{V}) \colon U }
$$
Here $U$ is the forgetful functor, 
and the left adjoint $F$ is obtained by taking the free colored operad generated by a collection.

By forgetting the symmetric group actions in the previous definitions we arrive at the notions of
\emph{non\nobreakdash-symmetric colored collections} and \emph{non\nobreakdash-symmetric colored operads}.
There is an evident forgetful functor from $C$\nobreakdash-colored operads to non\nobreakdash-symmetric $C$\nobreakdash-colored operads. 
Its left adjoint $\Sigma$ is defined by 
\begin{equation}
\label{leftadjoint}
(\Sigma \mathcal{O})(c_1,\ldots, c_n;c)
=
\coprod_{\sigma\in\Sigma_n}\mathcal{O}(c_{\sigma^{-1}(1)},\ldots, c_{\sigma^{-1}(n)}; c).
\end{equation}

Algebras over colored operads are defined as follows: 
Denote by $\mathcal{V}^C$ the product category of copies of $\mathcal{V}$ indexed by the colors $C$, 
i.e., 
$\mathcal{V}^C=\prod_{c\in C}\mathcal{V}$. 
For every object $X=(X(c))_{c\in C}$ in $\mathcal{V}^C$, 
the \emph{endomorphism colored operad} $\End({\mathbf{X}})$ is the $C$-colored operad defined by
$$
\End(\mathbf{X})(c_1,\ldots, c_n; c)
=
\Hom_{\mathcal{V}}(X(c_1)\otimes\cdots\otimes X(c_n),X(c)).
$$
By convention, the tensor product is $I$ when $n=0$. 
The composition product is the ordinary composition and the $\Sigma_n$-action is defined by permutation of the factors.

\begin{definition}
Let $\mathcal{O}$ be a $C$-colored operad in $\mathcal{V}$. 
An \emph{$\mathcal{O}$-algebra} (or an \emph{algebra over $\mathcal{O}$}) is an object $\mathbf{X}=(X(c))_{c\in C}$ of $\mathcal{V}^C$ 
together with a map
$
\mathcal{O}\to \End(\mathbf{X})
$
of $C$-colored operads.
\end{definition}

Equivalently, since the monoidal category $\mathcal{V}$ is closed, an $\mathcal{O}$-algebra can be defined as a family of objects $X(c)$ 
in $\mathcal{V}$ for every $c\in C$ together with maps
$$
\mathcal{O}(c_1,\ldots, c_n; c)\otimes X(c_1)\otimes\cdots \otimes X(c_n)\longrightarrow X(c)
$$
for every $(n+1)$-tuple $(c_1,\ldots,c_n, c)$, 
that are compatible with the symmetric group actions, 
the units of $\mathcal{O}$ and subject to the usual associativity relations.

\begin{remark}
\label{symmetrization}
If $\mathcal{O}$ is a non\nobreakdash-symmetric $C$\nobreakdash-colored operad, 
then $\mathcal{O}$\nobreakdash-algebras are defined in the same way by forgetting the symmetric group action on the endomorphism operad 
${\End}({\bf X})$.
With $\Sigma$ as in (\ref{leftadjoint}),
there is a bijective correspondence between the $\mathcal{O}$\nobreakdash-algebra structures and the $\Sigma \mathcal{O}$\nobreakdash-algebra 
structures on ${\bf X}\in\mathcal{V}^C$, 
for every non\nobreakdash-symmetric $C$\nobreakdash-colored operad $\mathcal{O}$.
\end{remark}

A \emph{map of $\mathcal{O}$-algebras} $\mathbf{f}\colon\mathbf{X}\to \mathbf{Y}$ is given by a family of maps $(f_c\colon X(c)\to Y(c))_{c\in C}$ 
such that the diagram of $C$-colored collections
$$
\xymatrix{
\mathcal{O}\ar[r] \ar[d] & \End(\mathbf{X}) \ar[d] \\
\End(\mathbf{Y}) \ar[r]& \Hom(\mathbf{X},\mathbf{Y})
}
$$
commutes, 
where the top and left arrows are the given $\mathcal{O}$-algebra structures on $\mathbf{X}$ and $\mathbf{Y}$, 
respectively, 
and the $C$-colored collection $\Hom(\mathbf{X},\mathbf{Y})$ is defined by 
$$
\Hom(\mathbf{X},\mathbf{Y}(c_1,\ldots, c_n;c))
=\Hom_{\mathcal{V}}(X(c_1)\otimes\cdots\otimes X(c_n), Y(c)).
$$
The right vertical and lower horizontal maps are induced by the family of maps $(f_c)_{c\in C}$.

If $\mathcal{V}$ has pullbacks, 
then a map $\mathbf{f}$ of $\mathcal{O}$-algebras can be viewed as a map of $C$-colored operads
$$
\mathcal{O}\longrightarrow \End(\mathbf{f}),
$$
where the $C$\nobreakdash-colored operad $\End(\mathbf{f})$ is the pullback of the diagram of $C$\nobreakdash-colored collections:
\begin{equation}
\xymatrix{
\End(\mathbf{f})\ar@{.>}[r] \ar@{.>}[d] & \End(\mathbf{X}) \ar[d] \\
\End(\mathbf{Y}) \ar[r]& \Hom(\mathbf{X},\mathbf{Y})  }
\label{end_f}
\end{equation}
Note that $\End(\mathbf{f})$ inherits a $C$-colored operad structure from the $C$-colored operads $\End(\mathbf{X})$ and $\End(\mathbf{Y})$. 
We denote the category of $\mathcal{O}$-algebras by $\Alg_{\mathcal{O}}(\mathcal{V})$.

\begin{definition}
\label{def:herdht}
For a $C$-colored operad $\mathcal{O}$ and object $\mathbf{X}=(X(c))_{c\in C}$ of $\mathcal{V}^C$ the \emph{restricted endomorphism operad} 
$\End_{\mathcal{O}}(\mathbf{X})$ is defined by
\begin{equation}
\End_{\mathcal{O}}(\mathbf{X})(c_1,\ldots, c_n; c)=
\left\{
\begin{array}{ll}
\End(\mathbf{X},\mathbf{Y})(c_1,\ldots, c_n;c) & \mbox{if }\mathcal{O}(c_1,\ldots, c_n;c)\ne 0, \\[0.2cm]
0 & \mbox{otherwise}.
\end{array}
\right.
\label{endp}
\end{equation}
\end{definition}
A motivation for Definition \ref{def:herdht} is the evident inclusion of $C$-colored operads $\End_{\mathcal{O}} (\mathbf{X})\to\End(\mathbf{X})$, 
which implies that every map $\mathcal{O}\to \End(\mathbf{X})$ of $C$-colored operad factors uniquely through the restricted endomorphism operad 
$\End_{\mathcal{O}}(\mathbf{X})$.
Hence an $\mathcal{O}$-algebra structure on $\mathbf{X}$ is tantamount to a map of $C$-colored operads $\mathcal{O}\to\End_{\mathcal{O}}(\mathbf{X})$. 
The same is also true for non-symmetric colored operads, 
as can be seen by using the non-symmetric version of the restricted endomorphism colored operad.
Similarly, 
if $\mathbf{X}$ and $\mathbf{Y}$ are objects of $\mathcal{V}^C$ and $\mathcal{O}$ is a $C$-colored operad, 
let $\Hom_{\mathcal{O}} (\mathbf{X},\mathbf{Y})$ denote the $C$-colored collection defined by  
\begin{equation}
\Hom_{\mathcal{O}}(\mathbf{X},\mathbf{Y})(c_1,\ldots, c_n; c)=
\left\{
\begin{array}{ll}
\Hom(\mathbf{X},\mathbf{Y})(c_1,\ldots, c_n;c) &\mbox{if }\mathcal{O}(c_1,\ldots, c_n;c)\ne 0,\\[0.2cm]
0 & \mbox{otherwise}. \\
\end{array}
\right.
\label{homp}
\end{equation}
For a map $\mathbf{f}\colon \mathbf{X}\to\mathbf{Y}$ of $\mathcal{O}$-algebras, 
denote by $\End_{\mathcal{O}}(\mathbf{f})$ the pullback (when it exists) of the restricted endomorphism operads of $\mathbf{X}$ and $\mathbf{Y}$ over 
$\Hom_{\mathcal{O}}(\mathbf{X},\mathbf{Y})$, as in~(\ref{end_f}).

If $\alpha\colon C\to D$ is a function between sets of colors, 
then any $D$-colored operad $\mathcal{O}$ pulls back to a $C$-colored operad $\alpha^*\mathcal{O}$ and there is a pair of adjoint functors 
(see~\cite[\S 1.6]{BM07}):
\begin{equation}
\label{coloradjoint1}
\xymatrix{
\alpha_{!} \colon {\Oper}_C(\mathcal{V})\ar@<4pt>[r] & \ar@<1pt>[l]{\Oper}_D(\mathcal{V}) \colon \alpha^* }
\end{equation}
The restriction functor $\alpha^*$ is defined as 
$$
(\alpha^* \mathcal{O})(c_1,\ldots,c_n;c)=\mathcal{O}(\alpha(c_1),\ldots, \alpha(c_n); \alpha(c)).
$$

A function $\alpha\colon C\to D$ defines also an adjoint pair between the corresponding categories of algebras for every for every $D$-colored operad 
$\mathcal{O}$ in $\mathcal{V}$:
$$
\xymatrix{
\alpha_{!} \colon {\Alg}_{\alpha^*\mathcal{O}}(\mathcal{V})\ar@<4pt>[r] & \ar@<1pt>[l]
{\Alg}_{\mathcal{O}}(\mathcal{V}) \colon\alpha^* }
$$
Here $\alpha^*$ is defined as follows: 
If ${\bf X}$ denotes an $\mathcal{O}$\nobreakdash-algebra with structure map $\gamma\colon \mathcal{O}\to \End({\bf X})$, 
then
$$
(\alpha^*{\bf X})(c)=X(\alpha(c))
$$
for all $c\in C$, with structure map defined by (\ref{coloradjoint1}),
namely
\begin{equation}
\label{alphastar}
\alpha^*\gamma\colon \alpha^*\mathcal{O}\longrightarrow \alpha^*\End({\bf X})=\End(\alpha^*{\bf X}).
\end{equation}

\subsection{Examples of colored operads}
\label{examples_operads}
If $C=\{c\}$, 
then a $C$-colored operad $\mathcal{O}$ is an ordinary operad, 
where we write $\mathcal{O}(n)$ instead of $\mathcal{O}(c,\ldots,c;c)$ with $n$ inputs, 
for $n\ge 0$. 
The \emph{associative operad} $\Ass$ is the one-color operad defined by $\Ass(n)=I[\Sigma_n]$ for $n\ge 0$, 
where $I[\Sigma_n]$ is the coproduct of copies of $I$ indexed by $\Sigma_n$. 
Note that $\Sigma_n$ acts freely on $I[\Sigma_n]$ by permutations. 
The \emph{commutative operad} $\Com$ is the one-color operad defined by $\Com(n)=I$ for $n\ge 0$. 
Algebras over $\Ass$ are associative monoids in $\mathcal{V}$, 
while algebras over $\Com$ are commutative monoids in $\mathcal{V}$.

\subsubsection{Modules over operad algebras}
\label{mulmod}
Let $\mathcal{O}$ be a (one\nobreakdash-colored) operad in $\mathcal{V}$ and $\modu_{\mathcal{O}}$ be the two-colored operad for $C=\{r,m\}$.
Its nonzero terms are
$$
\modu_{\mathcal{O}}(r,\stackrel{(n)}{\ldots}, r;r)=\mathcal{O}(n)
$$
for $n\ge 0$, and
$$
\modu_{\mathcal{O}}(c_1,\ldots, c_n;m)=\mathcal{O}(n)
$$
for $n\ge 1$, 
when exactly one $c_i$ is $m$ and the rest (if any) are equal to $r$. 
An algebra over $\modu_{\mathcal{O}}$ amounts to a pair $(R,M)$ of objects of $\mathcal{V}$, 
where $R$ is an $\mathcal{O}$\nobreakdash-algebra and $M$ a module over~$R$, 
i.e., 
an object equipped with a family of maps
$$
\mathcal{O}(n) \otimes R \otimes {\stackrel{(k-1)}{\cdots}} \otimes R \otimes M \otimes R \otimes {\stackrel{(n-k)}{\cdots}} \otimes R 
\longrightarrow 
M
$$
for 
$1\le k\le n$, 
subject to the usual equivariance, associativity and unit constraints.

If $\mathcal{O}=\Ass$, 
an algebra over $\modu_{\mathcal{O}}$ is a pair $(R,M)$, 
where $R$ is a monoid in~$\mathcal{V}$ and $M$ an $R$\nobreakdash-bimodule
($M$ acquires a right $R$\nobreakdash-action and a left $R$\nobreakdash-action that commute).
If $\mathcal{O}=\Com$, then $R$ is a commutative monoid in $\mathcal{V}$ and $M$ a module over $R$ (indistinctly left or right).
If $\alpha$ denotes the inclusion of $\{r\}$ into $\{r,m\}$, 
then $\alpha^*\modu_{\mathcal{O}} = \mathcal{O}$ for each operad $\mathcal{O}$, 
and $\alpha^*(R,M)=R$ on the level of the corresponding algebras.

As in \cite[\S 1.5.1]{BM07}, 
we note there exists a non\nobreakdash-symmetric colored operad for the notion of left modules. 
Suppose $\mathcal{O}$ is a non-symmetric (one-colored) operad and let $\LMod_{\mathcal{O}}$ be the non\nobreakdash-symmetric $C$\nobreakdash-colored 
operad with \hbox{$C=\{r,m\}$} defined by
$$
\LMod_{\mathcal{O}}(r,\stackrel{(n)}{\ldots}, r;r)=\mathcal{O}(n), \qquad \LMod_{\mathcal{O}}(r,\stackrel{(n)}{\ldots}, r,m;m)=\mathcal{O}(n+1)
$$
for $n\ge 0$, and zero otherwise. 
If $\mathcal{O}=\Ass$ (as a non\nobreakdash-symmetric operad),
the algebras over $\LMod_{\mathcal{O}}$ are pairs $(R, M)$ of objects of $\mathcal{V}$ where $R$ is a monoid and $M$ supports a left $R$-action.

\subsubsection{Maps of algebras over colored operads}
\label{mulmoralg}
Fix a $C$\nobreakdash-colored operad $\mathcal{O}$.
Let $D=\{0,1\}\times C$ and define the $D$\nobreakdash-colored operad $\Mor_{\mathcal{O}}$ by
$$
\Mor_{\mathcal{O}}((i_1,c_1),\ldots,(i_n,c_n);(i,c))=\left\{
\begin{array}{l} \mbox{$0$ if $i=0$ and $i_k=1$ for some $k$,} \\[0.2cm]
\mbox{$\mathcal{O}(c_1,\ldots, c_n;c)$ otherwise.}
\end{array}
\right.
$$
If ${\bf X}$ is an algebra over $\Mor_{\mathcal{O}}$, 
then ${\bf X}_0=(X(0,c))_{c\in C}$ and ${\bf X}_1=(X(1,c))_{c\in C}$ acquire $\mathcal{O}$\nobreakdash-algebra structures by restriction of colors.
The $\Mor_{\mathcal{O}}$\nobreakdash-algebra structure on $\bf X$ gives rise to a map of $\mathcal{O}$\nobreakdash-algebras 
${\bf f}\colon {\bf X}_0\longrightarrow {\bf X}_1$ as follows:
For $c\in C$, 
define $f_c\colon X(0,c)\to X(1,c)$ as the composite
\begin{equation}
\label{themap}
X(0,c)\longrightarrow \Mor_{\mathcal{O}}((0,c);(1,c))\otimes X(0,c)\longrightarrow X(1,c),
\end{equation}
where the first map is obtained by tensoring the unit $u_c\colon I\to \mathcal{O}(c;c)$ with $X(0,c)$.

Conversely, 
given $\mathcal{O}$\nobreakdash-algebras ${\bf X}_0$, ${\bf X}_1$ and a map of $\mathcal{O}$\nobreakdash-algebras ${\bf f}\colon{\bf X}_0\to{\bf X}_1$, 
there is a unique $\Mor_{\mathcal{O}}$\nobreakdash-algebra structure on ${\bf X}=\left(X_0(c),X_1(c)\right)_{c\in C}$ extending the given 
$\mathcal{O}$\nobreakdash-algebra structures and for which the map in~(\ref{themap}) coincides with $\bf f$.

\subsection{Semi model structure for colored operads}
Suppose $\mathcal{V}$ is a cofibrantly generated monoidal model category. 
The category of $C$-colored collections $\Coll_C(\mathcal{V})$ admits a cofibrantly generated model structure: 
$\mathcal{K}\to\mathcal{L}$ is a fibration or a weak equivalence if $\mathcal{K}(c_1,\ldots,c_n;c)\to \mathcal{L}(c_1,\ldots,c_n;c)$ 
is a fibration or a weak equivalence in $\mathcal{V}$, 
respectively, 
for every $(n+1)$-tuple of colors $(c_1,\ldots, c_n,c)$.  
This model structure can be transferred to a \emph{semi model structure} on $\Oper_C(\mathcal{V})$ using the free-forgetful 
adjunction:
\begin{equation}
\xymatrix{
F:\Coll_C(\mathcal{V}) \ar@<4pt>[r] & \ar@<1pt>[l] \Oper_C(\mathcal{V}):U }
\label{freeforgetful}
\end{equation}
A map is a fibration or a weak equivalence if its underlying map of $C$-colored collections is a fibration or a weak equivalence, 
respectively. 
The class of cofibrations is defined by the left lifting property with respect to the trivial fibrations.

In a semi model category the axioms of a model category hold except for the lifting and factorization axioms,
which are required to hold only for maps with cofibrant domain. 
Since the initial object of a semi model category is assumed to be cofibrant,  
trivial fibrations have the right lifting property with respect to all cofibrant objects. 
Semi model categories were first defined by Hovey in~\cite{Hov98}. 
The use of semi model structures for operads is due to Spitzweck~\cite{Spi}; 
see also~\cite[\S12]{Fre09}. 
A proof of the following result for operads with one color can be found in~\cite[Theorem~12.2.A]{Fre09} (cf.~\cite[Theorem 3.2]{Spi}). 
Our extension to colored operads can be proved using the same arguments.
\begin{theorem}
If $\mathcal{V}$ is a cofibrantly generated monoidal model category, 
the model structure on $\Coll_C(\mathcal{V})$ can be transferred via the free-forgetful adjunction (\ref{freeforgetful}) to a 
cofibrantly generated semi model structure on $\Oper_C(\mathcal{V})$, 
in which $\mathcal{O}\to\mathcal{O'}$ is a fibration or a weak equivalence if 
$\mathcal{O}(c_1,\ldots, c_n;c)\to \mathcal{O'}(c_1,\ldots, c_n; c)$ is a fibration or a weak equivalence in $\mathcal{V}$, 
respectively, 
for every $(c_1,\ldots, c_n, c)$.
\end{theorem}

\begin{remark}
If the unit of $\mathcal{V}$ is cofibrant and there exist a symmetric monoidal fibrant replacement functor and an interval with a coassociative 
and cocommutative comultiplication, 
then (\ref{freeforgetful}) defines a model structure on $\Oper_C(\mathcal{V})$ (see~\cite[Theorem~2.1]{BM07}). 
These conditions hold for the categories of topological spaces, 
simplicial sets and chain complexes, 
but not for symmetric spectra 
(see~\cite{GV10} for a discussion of the latter case).
\end{remark}

A $C$-colored operad $\mathcal{O}$ is \emph{$\Sigma$-cofibrant} if the underlying map of the unique map from the initial $C$-colored operad to 
$\mathcal{O}$ is a cofibration of collections. 
If $\mathcal{O}$ is $\Sigma$-cofibrant, 
then all unit maps $I\to\mathcal{O}(c;c)$ are cofibrations in $\mathcal{V}$. 
The categories of algebras over $\Sigma$-cofibrant $C$-colored operads admit transferred model structures by \cite[Theorem 12.3.A]{Fre09}, 
\cite[Theorem 4.7]{Spi}.

\begin{theorem}
Suppose $\mathcal{V}$ is a cofibrantly generated monoidal model category and $\mathcal{O}$ is a $\Sigma$\nobreakdash-cofibrant $C$-colored operad. 
Then the model structure on $\mathcal{V}^C$ can be transferred via the free-forgetful adjunction
$$
\xymatrix{
F_P \colon \mathcal{V}^C\ar@<4pt>[r] & \ar@<1pt>[l]\Alg_{\mathcal{O}}(\mathcal{V})\colon U_P }
$$
to a cofibrantly generated semi model structure on $\Alg_{\mathcal{O}}(\mathcal{V})$. 
\end{theorem}

Let $\mathcal{O}$ be a $C$-colored operad in $\mathcal{V}$. 
Denote by $\mathcal{O}_{\infty}$ a $\Sigma$-cofibrant replacement of $\mathcal{O}$ in the semi model category $\Oper_C(\mathcal{V})$.
That is, 
$\mathcal{O}_{\infty}$ is $\Sigma$-cofibrant and there is a trivial fibration $\mathcal{O}_{\infty}\to\mathcal{O}$.
As usual, 
$A_{\infty}$ is a $\Sigma$-cofibrant resolution of $\Ass$ and $E_{\infty}$ a $\Sigma$-cofibrant resolution of $\Com$.
A weak equivalence between $\Sigma$-cofibrant $C$-colored operads induces a Quillen equivalence between the categories of algebras; 
see \cite[Theorem 12.5.A]{Fre09} and \cite{Spi}.

\begin{theorem}
Let $\varphi \colon\mathcal{O}\to\mathcal{O'}$ be a weak equivalence between  $\Sigma$-cofibrant $C$-colored operads. 
Then the adjoint pair given by  extension and restriction
$$
\xymatrix{
\varphi_!\colon \Alg_{\mathcal{O}}(\mathcal{V})\ar@<4pt>[r] & \ar@<1pt>[l]\Alg_{\mathcal{O'}}(\mathcal{V})\colon \varphi^* }
$$
is a Quillen equivalence. 
\end{theorem}

Let $\mathcal{V}=\sSet$ and suppose $\mathcal{O}_1\to \mathcal{O}_2$ is a weak equivalence of $C$-colored operads in simplicial sets. 
For every tuple $(c_1,\ldots, c_n, c)$, 
we observe that $\mathcal{O}_2(c_1,\ldots, c_n;c)=0$ implies $\mathcal{O}_1(c_1,\ldots, c_n; c)=0$. 
In particular, 
this can be applied to $\mathcal{O}_{\infty}$ and $\mathcal{O}$.
We employ this fact in Section~\ref{hfgdfg}.

\subsection{$\mathcal{O}_{\infty}$-modules and $\mathcal{O}_{\infty}$-maps}
\label{inftymaps}
Suppose $\mathcal{O}$ is a (one\nobreakdash-colored) operad and $\mathcal{O}_{\infty}\to\mathcal{O}$ a $\Sigma$-cofibrant resolution. 
The $C$\nobreakdash-colored operad $\modu_{\mathcal{O}}$ is reviewed in Subsection~\ref{mulmod},
where $C=\{r,m\}$.
Let $\alpha$ denote the inclusion of $\{r\}$ into $C$.
If $(\modu_{\mathcal{O}})_{\infty}\to \modu_{\mathcal{O}}$ is a $\Sigma$-cofibrant resolution of $\modu_{\mathcal{O}}$, 
then
$$
\alpha^*(\modu_{\mathcal{O}})_{\infty}\longrightarrow \alpha^*\modu_{\mathcal{O}}=\mathcal{O}
$$
is a trivial fibration, since $\alpha^*$ preserves trivial fibrations.

If $\mathcal{O}_{\infty}$ is also cofibrant, 
then there is a lifting (unique up to homotopy) in the diagram:
\begin{equation}
\xymatrix{
& \alpha^*(\modu_{\mathcal{O}})_{\infty} \ar[d] \\
\mathcal{O}_{\infty} \ar@{.>}[ur] \ar[r] & \mathcal{O} }
\label{jajaja}
\end{equation}
If $(R,M)$ is an $(\modu_{\mathcal{O}})_{\infty}$\nobreakdash-algebra, 
then $R=\alpha^*(R,M)$ is an algebra over $\alpha^*(\modu_{\mathcal{O}})_{\infty}$ by (\ref{alphastar}), 
and hence an $\mathcal{O}_{\infty}$\nobreakdash-algebra by using the lifting in (\ref{jajaja}).
Although $M$ need not be a module,
we call it a \textit{$\mathcal{O}_{\infty}$\nobreakdash-module} over~$R$.

Now let $\mathcal{O}$ be a $C$\nobreakdash-colored operad with $\Sigma$-cofibrant resolution $\varphi\colon \mathcal{O}_{\infty}\to \mathcal{O}$.
The $D$-colored operad $\Mor_{\mathcal{O}}$ is reviewed in Subsection~\ref{mulmoralg}, 
where $D=\{0,1\}\times C$.
For $i\in\{0,1\}$, 
define the functions $\alpha_i\colon C\to D$ by $\alpha_0(c)=(0,c)$ and $\alpha_1(c)=(1,c)$. 
In both cases, $(\alpha_i)^*\Mor_{\mathcal{O}}=\mathcal{O}$. 
Thus, 
if $\Phi\colon \left(\Mor_{\mathcal{O}}\right)_{\infty}\to \Mor_{\mathcal{O}}$ is a $\Sigma$-cofibrant resolution of $\Mor_{\mathcal{O}}$ and 
$\mathcal{O}_{\infty}$ is cofibrant, 
there exist maps (in fact, weak equivalences) of $C$-colored operads
$$
\xymatrix{
& (\alpha_i)^*\left(\Mor_{\mathcal{O}}\right)_{\infty} \ar[d]^{(\alpha_i)^*\Phi} \\
\mathcal{O}_{\infty} \ar@{.>}[ur]^-{\lambda_i} \ar[r]^-{\varphi} & \mathcal{O} }
$$
for $i=0$ and $i=1$, unique up to homotopy, such that the triangle commutes.
Hence an algebra $\bf X$ over $\left(\Mor_{\mathcal{O}}\right)_{\infty}$ gives rise to a pair of $\mathcal{O}_{\infty}$\nobreakdash-algebras
$({\bf X}_0,{\bf X}_1)$ with additional structure linking them, 
which is weaker than a map of $\mathcal{O}_{\infty}$\nobreakdash-algebras. 
Specifically,
since the unit $I$ of $\mathcal{V}$ is cofibrant, we may choose, for each $c\in C$ a lifting in the diagram:
$$
\xymatrix{
& \left(\Mor_{\mathcal{O}}\right)_{\infty}((0,c);(1,c)) \ar[d] \\
I \ar[r]_-{u_c} \ar@{.>}[ur] & \Mor_{\mathcal{O}}((0,c);(1,c)) }
$$
Here $u_c$ is the map in (\ref{themap}). 
Although the lifting is not unique, 
it is unique up to homotopy. 
Therefore, 
the composite maps
$$
X(0,c)\longrightarrow \left(\Mor_{\mathcal{O}}\right)_{\infty}((0,c);(1,c))\otimes X(0,c)\longrightarrow X(1,c)
$$
yield a homotopy class of maps ${\bf X}_0\to{\bf X}_1$.
Each such map is called an \textit{$\mathcal{O}_{\infty}$\nobreakdash-map}.
Observe that if $\left(\Mor_{\mathcal{O}}\right)_{\infty}$ is cofibrant, 
there exists a lifting
$$
\xymatrix{
& \Mor_{\mathcal{O}_{\infty}} \ar[d] \\
\left(\Mor_{\mathcal{O}}\right)_{\infty} \ar[r]^-{\Phi} \ar@{.>}[ur]^{\Psi} & \Mor_{\mathcal{O}} }
$$
because the vertical arrow is a trivial fibration of $C$\nobreakdash-colored operads. 
Hence every map of $\mathcal{O}_{\infty}$\nobreakdash-algebras admits an $\mathcal{O}_{\infty}$\nobreakdash-map structure.
%

\section{Bousfield localizations and colocalizations} \label{ytrgvv}
We recall from \cite{Barwick}, \cite{Hirschhorn} the notions of left and right Bousfield localizations, 
here referred to as \emph{Bousfield localizations} and \emph{Bousfield colocalizations}, respectively. 
By a \emph{combinatorial} model category, 
we mean a locally presentable and cofibrantly generated model category.
Suppose $\mathcal{M}$ is a simplicial combinatorial model category.
Let $\map(X,Y)=\Map(QX, RY)$ denote a homotopy function complex, 
where $\Map(-,-)$ denotes the simplicial enrichment, 
$Q$ a functorial cofibrant replacement and $R$ a functorial fibrant replacement in $\caM$.

Let $\mathscr{S}$ be a class of maps in $\mathcal{M}$. 
An object $Z$ of $\mathcal{M}$ is \emph{$\mathscr{S}$-local} if it is fibrant and the induced map of function complexes
$$
f^*\colon\map(B,Z)\longrightarrow \map(A, Z)
$$
is a weak equivalence of simplicial sets for every map $f\colon A\to B$ in $\mathscr{S}$.

A map $g\colon X\to Y$ is an \emph{$\mathscr{S}$-local equivalence} if the induced map
$$
g^*\colon\map(Y,Z)\longrightarrow \map(X, Z)
$$
is a weak equivalence of simplicial sets for every $Z$ that is $\mathscr{S}$-local.

Let $\mathscr{K}$ be a class of objects in $\mathcal{M}$.
A map $g\colon X\to Y$ is a \emph{$\mathscr{K}$-colocal equivalence} if 
$$
g_*\colon\map(Z,X)\longrightarrow \map(Z, Y)
$$
is a weak equivalence of simplicial sets for every $Z$ in $\mathscr{K}$.

An object $Z$ of $\mathcal{M}$ is \emph{$\mathscr{K}$-colocal} if it is cofibrant and 
$$
f_*\colon\map(Z,A)\longrightarrow \map(Z, B)
$$
is a weak equivalence of simplicial sets for every $\mathscr{K}$-colocal equivalence $f\colon A\to B$.

\begin{definition}
\label{Bloc}
The Bousfield localization (if it exists) of $\mathcal{M}$ with respect to $\mathscr{S}$ is a simplicial model structure 
$L_{\mathscr{S}}\mathcal{M}$ on $\mathcal{M}$ such that:
\begin{itemize}
\item[\rm (i)] The cofibrations in $L_{\mathscr{S}}\mathcal{M}$ are the same as the cofibrations in $\mathcal{M}$.
\item[\rm (ii)] The weak equivalences in $L_{\mathscr{S}}\mathcal{M}$ are the $\mathscr{S}$-local equivalences.
\item[\rm (iii)] The fibrant objects in $L_{\mathscr{S}}\mathcal{M}$ are the $\mathscr{S}$-local objects.
\end{itemize}
\end{definition}

The $\mathscr{S}$-localization functor in $\mathcal{M}$ is defined as a functorial fibrant replacement $(L,\eta)$ in $L_{\mathscr{S}}\mathcal{M}$. 
Thus, 
for every $X$ in $\mathcal{M}$, 
the natural map $\eta_X\colon X\to LX$ is a cofibration and an $\mathscr{S}$-local equivalence, 
and $LX$ is $\mathscr{S}$-local.

\begin{definition}
\label{Bcoloc}
The Bousfield colocalization (if it exists) of $\mathcal{M}$ with respect to $\mathscr{K}$ is a simplicial model structure 
$C^{\mathscr{K}}\mathcal{M}$ on $\mathcal{M}$ such that:
\begin{itemize}
\item[\rm (i)] The fibrations in $C^{\mathscr{K}}\mathcal{M}$ are the same as the fibrations in $\mathcal{M}$.
\item[\rm (ii)] The weak equivalences in $C^{\mathscr{K}}\mathcal{M}$ are the $\mathscr{K}$-colocal equivalences.
\item[\rm (iii)] The cofibrant objects in $C^{\mathscr{K}}\mathcal{M}$ are the $\mathscr{K}$-colocal objects.
\end{itemize}
\end{definition}

The $\mathscr{K}$-colocalization functor in $\mathcal{M}$ is defined as a functorial cofibrant replacement $(K,\varepsilon)$ in 
$C^{\mathscr{K}}\mathcal{M}$. 
Thus, 
for every $X$ in $\mathcal{M}$, 
the natural map $\varepsilon_X\colon KX\to X$ is a fibration and a $\mathscr{K}$-colocal equivalence, and $KX$ is $\mathscr{K}$-colocal.

\begin{remark}
Bousfield localizations always exist if $\mathcal{M}$ is left proper and $\mathscr{S}$ is a set, 
and dually, 
Bousfield colocalizations exist if $\mathcal{M}$ is right proper and $\mathscr{K}$ is a set; 
see \cite{Barwick}, \cite{Hirschhorn}. 
Note that if $\mathcal{M}$ is a monoidal model category, 
then although $L_{\mathscr{S}}\mathcal{M}$ and $C^{\mathscr{K}}\mathcal{M}$ are monoidal as categories, 
they are not necessarily \emph{monoidal model} categories since the pushout\nobreakdash-product axiom does not hold in general.
\end{remark}
{\bf Acknowledgements.}
We gratefully acknowledge hospitality and support from IMUB at the Universitat de Barcelona in the framework of the NILS mobility project. 
The authors are partially supported by the RCN 185335/V30. 
The first author was supported by the MEC-FEDER grant MTM2010-15831 and by the Generalitat de Catalunya as a member of the team 2009 SGR 119.
\def\cprime{$'$}
\begin{small}

\begin{center}
Departament d'\`Algebra i Geometria, Universitat de Barcelona, Spain.\\
e-mail: jgutierrez@crm.cat
\end{center}
\begin{center}
Mathematisches Institut, Universit\"at Osnabr\"uck, Germany.\\
e-mail: oroendig@mathematik.uni-osnabrueck.de
\end{center}
\begin{center}
Fakult\"at f\"ur Mathematik, Universit\"at Regensburg, Germany.\\
e-mail: Markus.Spitzweck@mathematik.uni-regensburg.de
\end{center}
\begin{center}
Department of Mathematics, University of Oslo, Norway.\\
e-mail: paularne@math.uio.no
\end{center}
\end{small}
\end{document}